\documentclass[article, 11pt]{amsart}
\usepackage{amsmath, amssymb, fontenc, amsthm, array, tikz, hyperref, amsfonts, enumitem, mathtools, stmaryrd, centernot}
\usepackage[english]{babel}
\usepackage[toc]{appendix}
\usepackage{bigfoot}
\usepackage{doi}
\usepackage[margin=1in]{geometry}
\usetikzlibrary{arrows}

\tikzset{degil/.style={
            decoration={markings,
            mark= at position 0.5 with {
                  \node[transform shape] (tempnode) {$\subseteq$};
                  }
              },
              postaction={decorate}
}
}

\renewcommand{\restriction}{\mathord{\upharpoonright}}

\makeatletter
\newcommand{\preceqdot}{\mathrel{\mathpalette\pr@ceqd@t\relax}}
\newcommand{\pr@ceqd@t}[2]{%
  \begingroup
  \sbox\z@{$#1\prec$}\sbox\tw@{$#1\preceq$}%
  \dimen@=\dimexpr\ht\tw@-\ht\z@\relax
  {\preceq}%
  \mkern-5mu
  \raisebox{\dimen@}{$\m@th#1\cdot$}%
  \endgroup
}
\makeatother

\providecommand{\keywords}[1]{  \small  \textbf{\textit{Keywords:}} #1 \normalsize}

\hypersetup{
    colorlinks = true,
    citecolor = blue,
    urlcolor = blue,
    linkcolor = red
}

\theoremstyle{plain}
\newtheorem{theorem}{Theorem}[section]
\newtheorem{lemma}[theorem]{Lemma}
\newtheorem{corollary}[theorem]{Corollary}
\newtheorem{proposition}[theorem]{Proposition}
\theoremstyle{definition}
\newtheorem{definition}[theorem]{Definition}

\newtheorem{observation}[theorem]{Observation}


\begin{document}

\title{Direct powers, ultrapowers and cumulative powers}
\author{Pedro Teixeira Yago}
\address{Orcid: \href{}{0000-0001-7993-4516}, Classe di Lettere e Filosofia, Scuola Normale Superiore di Pisa, Italy}
\email{pedro.tyago@outlook.com}
\date{}

\begin{abstract}
In this paper we investigate cumulative direct powers of functions on structures, or \emph{cumulative powers}, and study their properties. Particularly, we show how they extend the preservation phenomena of reduced powers, direct powers and ultrapowers by offering a characterization of the fragment of first-order theory it preserves, and elucidate the connections between the three sorts of constructions. More precisely, we show how both direct powers and ultrapowers may be obtained from cumulative powers as quotients by appropriate equivalence relations. We further offer a construction of Conway's surreal field by using direct limits of hierarchies of ultrapowers.
\end{abstract}

\maketitle


\section{Introduction}\label{introduction}

Functional constructions on mathematical structures play a major role in model theory in the form of direct powers and ultrapowers. These constructions are essentially non-cumulative. Despite there always being a trivial embedding from the generating structure into its direct power or ultrapower, there is nevertheless no direct interaction between the elements from the generating structure and the resulting ones. From an algebraic point of view, the alternative, power structures which retain the objects of their generating structure, is a natural generalization of those ubiquitous model theoretic tools. Nevertheless, so far, it has not received much attention. In this paper, we seek to start to remedy that.

Consider the collection of all functions from a given index set $I$ into the naturals, $\mathbb{N}^I$. There is a straightforward manner of extending any operation $\mathrm{F}$ on $\mathbb{N}$ to an $\mathrm{F}'$ on $\mathbb{N}^I$ that behaves similarly to the original operation, as in the corresponding direct power -- that is, pointwisely, so that $\mathrm{F}'(f_1, ..., f_n) = g$ iff for all $j \in I$, $\mathrm{F} \big{(} f_1(j), ..., f_n(j) \big{)} = g(j)$. Likewise, we might define the constant function $\overline{0}$ as the identity element of the pointwise addition, and its image under the extended successor function is precisely the identity element of the extended multiplication. We may extrapolate this construction, and given a family of index sets $\{I_n\}_{n \in \omega}$, by inductively defining $\mathbb{N}_n$ by setting $\mathbb{N}_0 = \mathbb{N}$ and $\mathbb{N}_{n+1} = (\mathbb{N}_n)^{I_n}$, we may keep extending those operations and identity elements accordingly. Since there is a natural embedding between any two structures defined in this way, we may even take $\mathbb{N}_\omega$ to be the direct limit, and given a proper class sized family of index sets $\{I_\alpha\}_{\alpha \in \mathbf{On}}$, keep transfinitely extending this hierarchy. Each of the $\mathbb{N}_\beta$ inherit nice properties of $\mathbb{N}$, and given the preservation results of direct powers by Alfred Horn \cite{horn} and Joseph Weinstein \cite{weinstein}, it may be checked that they are indeed semirings. We may do the same thing with $\mathbb{Z}$, and similarly get a hierarchy of rings $\mathbb{Z}_\beta$.

Suppose now we take a cumulative counterpart of the above structures. Starting with $\mathbb{N}_0 = \mathbb{N}$, we take $\mathbb{N}_{\beta^+} = (\mathbb{N}_\beta)^{I_\beta} \cup \mathbb{N}_\beta$. Then, the operations and constants may no longer be pointwisely defined. Call such a construction a \emph{cumulative (direct) power}. Is there a natural way of extending addition and multiplication to those structures? If so, what first-order properties are retained by cumulative powers? 

The present paper has three different goals, interrelated within a broad study of cumulative powers and hierarchical power constructions. The first is to answer the above questions. To do so, in Section 2, we define a framework for working with hierarchies of cumulative powers of arbitrary structures. The main feature of the construction, offered in Definition \ref{def hierarchy}, is a hereditary definition of the extended operations and relations of higher levels, so that they are reduced to the lower levels at which they are originally defined. The resulting structures have elements from different levels the hierarchy, such as higher-order functions and non-functional elements, which nevertheless interact in a natural manner. We then show how a substantial fragment of first-order theory is preserved by the constructions by offering the precise fragment of first-order theory preserved by cumulative powers. To present it, let us call a sentence preserved by direct powers a \emph{direct power sentence} (a characterization of which is offered in \cite{weinstein}). Let a formula be \emph{non-collapsible} if its positive equalities do not force the collapse of different levels of the hierarchy of cumulative powers -- a property which, as given in Definition \ref{def non collapsible 2}, may be syntactically defined. Then:

\setcounter{section}{2}
\setcounter{theorem}{25}

\begin{corollary}
Let $\varphi \in \mathcal{L}^\sigma$ be a sentence. Then:

\begin{itemize}
\item[\emph{(1)}] $\varphi$ is preserved by cumulative powers iff it is equivalent to a non-collapsible direct power sentence; 
\item[\emph{(2)}] $\varphi$ is preserved by arbitrary cumulative powers iff it is equivalent to a constant-free non-collapsible direct power sentence.
\end{itemize}
\end{corollary}

\setcounter{section}{1}
\setcounter{theorem}{0}

\noindent
This extends the well known results of preservation of first-order properties by reduced products by \cite{changmorel}, \cite{keisler1965}, and \cite{galvin1970}, of direct products by \cite{weinstein}, and the classical  \L o\'{s}' theorem for ultraproducts \cite{los}. These preservation results may be thus summarized as follows. Let $\mathcal{L}_\mathcal{F}$ be the fragment of first-order logic preserved by cumulative powers, $\mathcal{L}_{\Pi / F}$ the fragment preserved by reduced powers, $\mathcal{L}_\Pi$ the fragment preserved by direct powers, and $\mathcal{L}_{\Pi / \mathcal{U}}$ the fragment preserved by ultrapowers. Then, we have the following graph of containment of preserved first-order formulas:

\begin{center}
$\mathcal{L}_\mathcal{F}, \mathcal{L}_{\Pi / F} \subseteq \mathcal{L}_{\Pi} \subseteq \mathcal{L}_{\Pi / \mathcal{U}}$
\end{center}

Given the construction of cumulative powers, the second aim of this paper naturally arises, which is how the new construction relates to direct powers and ultrapowers. As we argue in Section 3, the main difference between cumulative powers and direct powers is that in the latter equality is also defined hereditarily, whereas in the former, equality is defined as identity. This asymmetry ultimately accounts for the difference between the preservation results for direct powers and the one presented in Corollary \ref{corollary characterization cumulative 2}. This fact is consolidated by showing, in Theorem \ref{theorem isomorphism non cum hierarchy}, how taking a quotient of a cumulative power of a structure by an equivalence relation, defined by an equality modulo hereditary identity, results in a structure which is isomorphic to its relative direct power.

In Section 4, we show a cumulative power may be further refined by taking a quotient of it by yet another equivalence relation, given by an adequately defined ultrafilter, so that the resulting structure is isomorphic to a relative ultrapower, and therefore preserves any first-order property of a finite language. In this way, we may see how cumulative powers serve as a generalisation of both ultrapower and direct power constructions, as the choice of the equivalence relation which quotients it defines the first-order faithfulness to the original structure. We may appreciate the connection between the different structures in Figure 1.

\Small

\begin{figure}

\begin{tikzpicture}[scale=2, label distance=2pt, shorten >=1pt, shorten <=1pt]

 	\node (1) at (0,0) [circle]{$\mathfrak{C}(\mathcal{A}, I)$};
 	\node (2) at (3.5,0) [circle]{$\mathcal{A}^I$};
	\node (3) at (3.5,-3.5) [circle]{$\mathcal{A}^I / \mathcal{U}$};
  
	\draw [->] (1) -- (2) node[midway,above, sloped] {$\text{equivalence modulo hereditary identity}$};
	\draw [->] (2) -- (3) node[midway,above, sloped] {$\text{equivalence modulo}\ \mathcal{U}$};
	\draw [->] (1) -- (3) node[midway,above, sloped]  {$\text{equivalence modulo}\ \mathcal{U}\ \text{of hereditarily identified elements}$};

\end{tikzpicture}

\caption{The connection between the power structures, for $I$ an index set, $\mathcal{U}$ an ultrafilter over $I$, $\mathcal{A}$ a structure, and $\mathfrak{C}(\mathcal{A}, I)$ the cumulative power generated by them.}
\end{figure}

\normalsize

As the last goal, we offer an application of the hierarchical constructions to real closed fields. We show that, at each ordinal stage of either a hierarchy generated by ultrapowers or quotiented cumulative powers, a suitable choice of ultrafilters produces hyperreal fields, whereas the hierarchy as a whole, under the assumption of Global Choice -- following results by Conway \cite{conway2000} and Ehrlich \cite{ehrlich1988} --, forms a real closed field isomorphic to the surreal numbers.

This paper offers cumulative power hierarchies as independent objects of study, with systematic analyses of their logical behavior and their relations with more usual functional constructions. As we shall see, the cumulative power of a given structure produces structures which are not, in general, isomorphic to ultrapowers or direct powers of that structure. Given their model theoretic invariances, as such, the subject should be of algebraic interest. We believe the results offer a novel perspective on classical power constructions and open several avenues for further research in model theory and its applications.

The background theory in which the paper is developed may be taken to be ZFC, as many proofs make use of Choice. However, sometimes we might consider proper-class sized structures, and some results make the explicit assumption of Global Choice, at which point they may be interpreted as having NBG as their background.

\subsection{Conventions}

Before we begin, we start with a few remarks on conventions we shall use. For $n \in \omega$, a set $A$ and $a \in A$, we write $\langle x_1, ..., a, ..., x_{n-1} \rangle$ for an $n$-tuple of $A^n$ in which $a$ occurs in some position. We shall also follow the convention of denoting structures by a calligraphic font, and their corresponding domains, by an italicized font -- for example, $\mathcal{A}$ and $A$, respectively -- although, in some cases -- when the notation becomes heavier --, we may use the structure's name to denote its domain -- for example, for a structure $\mathfrak{C}(\mathcal{A}, \mathbf{I})$, we shall write $a \in \mathfrak{C}(\mathcal{A}, \mathbf{I})$ to mean $a$ is a member of its domain. Let $\mathcal{A}$ be a structure and $I$ be a set. We denote by $\mathcal{A}^I$ the direct power of $A$ with index set $I$, and by $A^I$ its domain, that is, the set of all functions from $I$ into $A$. If $\mathcal{U}$ is an ultrafilter over $I$, we denote by $\mathcal{A}^I / \mathcal{U}$ the resulting ultrapower. By the \emph{natural embedding} of a set or structure into its direct power or ultrapower, we mean the mapping which takes each element of the domain of the set and maps it to either the constant function which assigns to every element of the index set the element of the domain in question, in the case of direct powers, or its corresponding equivalence class modulo the ultrafilter, in the case of ultrapowers. We let $\mathbf{On}$ stand for the class of ordinals, $\alpha$, $\beta$, $\delta$, $\lambda$ and so on, stand for ordinals, $\kappa$, $\tau$, $\xi$, $\zeta$, and so on, for cardinals, $\varphi$, $\psi$, $\gamma$ and so on, for formulas, $\Delta$ for an arbitrary string of quantifiers, and $t$ and $s$ for arbitrary terms of a given language. For an ordinal $\beta$ or a cardinal $\kappa$, we denote by $\beta^+$ and $\kappa^+$ their respective successors.

We take the basic first-order language to be composed of individual variables, negation, conjunction, equality, and the universal quantifier (where the other connectives and quantifiers may be defined as usual).  For an ordered set $\mathcal{A}$, we denote by $\mathrm{coi}(\mathcal{A})$ and $\mathrm{cof}(\mathcal{A})$ its coinitiality and cofinality, respectively. For a signature $\sigma$, we denote by $\mathcal{L}^\sigma$ the first-order language induced by it. As usual, we use a calligraphic letter $\mathcal{A}$ to denote a structure, and the respective italicized letter $A$ to denote its corresponding domain. Given a constant $\mathrm{c}$, function symbol $\mathrm{F}$ or relation symbol $\mathrm{R}$ of $\sigma$, and a $\sigma$-structure $\mathcal{A}$, we denote by $\mathrm{c}^\mathcal{A}$, $\mathrm{F}^\mathcal{A}$ and $\mathrm{R}^\mathcal{A}$ their interpretations in $\mathcal{A}$, respectively. Whenever we are working with some model, we shall state things like $\mathcal{A} \models \varphi[a_1, ..., a_n]$, for $a_0, ..., a_n \in A$, to mean $\varphi$ is a formula with at most $n$ free variables $x_1, ..., x_n$, and $\varphi$ is true in $\mathcal{A}$ when the $x_i$ are assigned the respective values $a_i$. Likewise, given a $\sigma$-term $t$, we shall denote by $t^\mathcal{A}(a_1, ..., a_n)$ the element of $\mathcal{A}$ obtained by interpreting $t$ with each of its free variables $x_i$ being assigned the value $a_i$ -- for example, in $\mathbb{N}$, $x[2] + y[3]$ is the element $2 +^\mathbb{N} 3 = 5$.

\section{A hierarchy of cumulative direct powers}\label{sec modelling}

In this section we introduce the main object of investigation: direct power constructions which retain the elements of their generating structures, and resulting hierarchies of functions on algebraic structures. The idea is simple: we iterate the operation of taking all functions from a set to itself, and collect the resulting objects together with the original ones.

\begin{definition}\label{def function hier}
Let $X$ be a designated set, and $\mathbf{I} = \{I_\alpha\}_{\alpha < \gamma}$ be a family of index sets. We then define, for $\beta, \lambda < \gamma$:

\smallskip

\begin{itemize}
\item $\mathcal{F}_0(X, \mathbf{I}) = X$;
\item $\mathcal{F}_{\beta^+}(X, \mathbf{I}) = \big{(} \mathcal{F}_\beta(X, \mathbf{I}) \big{)}^{I_\beta} \cup \mathcal{F}_\beta(X, \mathbf{I})$;
\item $\mathcal{F}_{\lambda}(X, \mathbf{I}) = \bigcup_{\alpha < \lambda} \mathcal{F}_\alpha(X, \mathbf{I})$, for limit $\lambda$.
\end{itemize}

\smallskip

\noindent
If $\mathbf{I}$ is proper class sized, that is, $\mathbf{I} = \{I_\alpha\}_{\alpha \in \mathbf{On}}$, we let $\mathcal{F}_\mathbf{On}(X, \mathbf{I}) = \bigcup_{\alpha \in \mathbf{On}} \mathcal{F}_\alpha(X, \mathbf{I})$.
\end{definition}

\medskip

Notice, at limit levels of the construction, no new objects are created. We can immediately define a notion of level for the elements of the hierarchy.

\begin{definition}[Rank of an element]\label{def level}
Let $a \in \mathcal{F}_\mathbf{On}(X, \mathbf{I})$. Then, define the \emph{rank $\rho$ of $a$} as $\mathrm{min}\{\beta \mid a \in \mathcal{F}_\beta(X, \mathbf{I})\}$.
\end{definition}

In other words, for $a \in \mathcal{F}_\mathbf{On}(X, \mathbf{I})$, $\rho(a)$ is the level at which $a$ is generated. For clarity of our presentation, we shall also define $\mathcal{F}_\beta^*(X, \mathbf{I}) = \mathcal{F}_\beta(X, \mathbf{I}) \setminus \bigcup_{\alpha < \beta} \mathcal{F}_\alpha(X, \mathbf{I})$.\footnote{Notice $\mathcal{F}_\beta^*(X, \mathbf{I}) = \varnothing$ iff either $\beta > 0$ is a limit ordinal or $\beta = 0$ and $X = \varnothing$.} Notice $a \in \mathcal{F}_\beta^*(X, \mathbf{I})$ iff $a \in \mathcal{F}_\beta(X, \mathbf{I})$ and $\rho(a) = \beta$, so that we may use the former to abbreviate the latter. Furthermore, for a successor $\alpha > \rho(a)$, we shall denote the constant function $a : I_{\alpha-1} \to \mathcal{F}_{\alpha-1}(A, \mathbf{I}); j \mapsto a$ by $\overline{a}^\alpha$. If the ordinal $\alpha$ is clear by context, we omit the superscript and write $\overline{a}$.

\begin{definition}[Hereditary function]
Let a functional hierarchy be generated by the set $X$ and family of index sets $\mathbf{I} = \{I_\alpha\}_{\alpha \in \mathbf{On}}$. The \emph{hereditary function} $\vartheta: (\mathbf{On} \times \mathcal{F}_\mathbf{On}(X, \mathbf{I}) \times \bigcup \mathbf{I}) \to \mathcal{F}_\mathbf{On}(A, \mathbf{I})$ is the function such that, for each $a \in \mathcal{F}_\mathbf{On}(X, \mathbf{I})$, $j \in \bigcup \mathbf{I}$ and $\alpha \in \mathbf{On}$:

\begin{center}
$\vartheta^\alpha_j(a) = \begin{cases}
a(j),\ \text{if}\ \rho(a) \geq \alpha\ \&\ j \in \mathrm{dom}(a)\\
a,\ \text{otherwise}
\end{cases}$
\end{center}

\noindent
In case $\mathbf{I} = \{I\}$, we define the simpler function

\begin{center}
$\vartheta_j(a) = \begin{cases}
a(j),\ \text{if}\ \mathrm{dom}(a) = I\\
a,\ \text{otherwise}
\end{cases}$
\end{center}
\end{definition}

\medskip

Notice that, despite the hereditary function $\vartheta$ being defined on proper classes, we do not need proper classes to define the structures above. For each $\beta \in \mathbf{On}$, the function $\vartheta \restriction_{\beta}$ is not a proper class, and suffices for defining the structure $\mathfrak{C}_\beta(\mathcal{A}, \mathbf{I})$. 

When $X$ has a structure (i.e. functions, relations and constants defined on it) we want to extend them to the whole hierarchy generated by it. That motivates the definition of a power structure which is cumulative, that is, where elements of the generating structure coexist with the generated functional elements of a direct power.

\begin{definition}[Cumulative direct power]\label{def cumulative hierarchy}
Let $\sigma = \langle \{F_i\}_{i \in J}, \{R_i\}_{i \in K}, \{c_i\}_{i \in L}, \mathrm{ar} \rangle$ be a signature.\footnote{Where $\mathrm{ar}: (\{F_i\}_{i \in J} \cup \{R_i\}_{i \in K}) \to \mathbb{N}$ is the arity function.} Let also $I$ be an index set, $\mathcal{A} = \langle A, \langle \mathrm{F}_i \rangle_{i \in J}, \langle \mathrm{R}_i \rangle_{i \in K}, \langle \mathrm{c}_i \rangle_{i \in L} \rangle$ be a $\sigma$-structure, where $A$ is the underlying set, $\mathrm{F}_i: A^{\mathrm{ar}(F_i)} \to A$ a $\mathrm{ar}(F_i)$-nary function, $\mathrm{R}_i \subseteq A^{\mathrm{ar}(R_i)}$ a $\mathrm{ar}(R_i)$-ary relation, and $\mathrm{c}_i \in A$ the interpretation of the constant $c_i$. The \emph{cumulative direct power of $\mathcal{A}$ by $I$} (shortly, \emph{cumulative power}), denoted by $\mathfrak{C}(\mathcal{A}, I)$, is the structure $\langle \mathcal{F}_1(A, \{I\}), \langle \mathrm{F}^+_i \rangle_{i \in J}, \langle \mathrm{R}^+_i \rangle_{i \in K}, \langle \mathrm{c}^+_i \rangle_{i \in L} \rangle$, where $\mathrm{F}^+_i : (\mathcal{F}_1(A, \{I\}))^{\mathrm{ar}(F_i)} \to \mathcal{F}_1(A, \{I\})$, $\mathrm{R}^+_i \subseteq (\mathcal{F}_1(A, \{I\}))^{\mathrm{ar}(R_i)}$ and $\mathrm{c}^+_i \in \mathcal{F}_1(A, \{I\})$ are defined in the following way.

\begin{itemize}
\item[(i)] $\mathrm{F}^+_i(a_1, ..., a_{\mathrm{ar}(\mathrm{F}_i)}) = b$ iff $\forall j \in I \Big{(} \mathrm{F}_i \big{(} \vartheta_j(a_1), ..., \vartheta_j(a_{\mathrm{ar}(\mathrm{F}_i)}) \big{)} = \vartheta_j(b) \Big{)}$ and 
$\mathrm{max}\{\rho(a_1), ..., \rho(a_{\mathrm{ar}(\mathrm{F}_i)})\} = \rho(b)$;
\item[(ii)] $\mathrm{R}^+_i a_1 ... a_{ar(R_i)}$ iff $\forall j \in I \big{(} \mathrm{R}_i \vartheta_j(a_1) ... \vartheta_j(a_{\mathrm{ar}(\mathrm{R}_i)}) \big{)}$;
\item[(iii)] $\mathrm{c}^+_i = \overline{\mathrm{c}_i}$.
\end{itemize}
\end{definition}

\medskip

Notice $\mathfrak{C}(\mathcal{A}, I)$ is an extension of $\mathcal{A}$.

There are two ways of extending the cumulative power construction to a hierarchical setting. The first, is by iterating the cumulative power construction to obtain successor levels of the hierarchy, and take the direct limit at limit stages.

\begin{definition}[Cumulative power hierarchy]\label{def hierarchy 1}
Let $\sigma = \langle \{F_i\}_{i \in I}, \{R_i\}_{i \in J}, \{c_i\}_{i \in K}, \mathrm{ar} \rangle$ be a signature. Let also $\{I_\alpha\}_{\alpha \in \mathbf{On}}$ be a family of sets, $\mathcal{A} = \langle A, \langle \mathrm{F}_i \rangle_{i \in J}, \langle \mathrm{R}_i \rangle_{i \in K}, \langle \mathrm{c}_i \rangle_{i \in L} \rangle$ be a $\sigma$-structure. The \emph{cumulative power hierarchy of $\mathcal{A}$ generated by $\mathbf{I} = \{I_\alpha\}_{\alpha \in \mathbf{On}}$} is composed by the following hierarchy of structures, for each $\beta \in \mathbf{On}$:

\begin{itemize}
\item $\mathfrak{C}_0(\mathcal{A}, \mathbf{I}) = \mathcal{A}$;
\item $\mathfrak{C}_{\beta^+}(\mathcal{A}, \mathbf{I}) = \mathfrak{C}(\mathfrak{C}_\beta(\mathcal{A}, \mathbf{I}), I_\beta) = \langle \mathcal{F}_{\beta^+}(A, \mathbf{I}), \langle \mathrm{F}_{i, \beta^+} \rangle_{i \in J}, \langle \mathrm{R}_{i, \beta^+} \rangle_{i \in K}, \langle \mathrm{c}_{i, \beta^+} \rangle_{i \in L} \rangle$;
\item $\mathfrak{C}_\lambda(\mathcal{A}, \mathbf{I}) = \bigcup_{\alpha < \lambda} \mathfrak{C}_\alpha(\mathcal{A}, \mathbf{I}) /\! \approx_\lambda$, for a limit $\lambda$, is the direct limit of the system $\langle \{\mathfrak{C}_\alpha(\mathcal{A}, \mathbf{I})\}_{\alpha < \lambda}, e^\alpha_\beta \rangle$, with the embeddings inductively defined by

\begin{itemize}
\item $e^\alpha_\alpha$ is the identity,
\item $e^\alpha_{\beta^+} : \mathfrak{C}_\alpha(\mathcal{A}, \mathbf{I}) \to \mathfrak{C}_{\beta^+}(\mathcal{A}, \mathbf{I}); x \mapsto \overline{e^\alpha_\beta(x)}$,
\item $e^\alpha_\lambda : \mathfrak{C}_\alpha(\mathcal{A}, \mathbf{I}) \to \mathfrak{C}_\lambda(\mathcal{A}, \mathbf{I}); x \mapsto [x]_\lambda$,
\end{itemize}

\noindent
where $[\cdot]_\lambda$ denotes the equivalence defined by the relation $\approx_\lambda$, given by $x \approx_\lambda y$ iff there are $\alpha \leq \beta < \lambda$ such that either $e^\alpha_\beta(x) = y$ or $e^\alpha_\beta(y) = x$
\end{itemize}
\end{definition}

\medskip

Notice, in the above case, from the level $\omega$ on, the elements of the cumulative power hierarchy are no longer elements of the functional cumulative hierarchy $\mathcal{F}_\mathbf{On}(A, \mathbf{I})$, so that their rank are no longer defined by Definition \ref{def level}. Instead, similarly to before, for an element $a$ of the cumulative power hierarchy, we shall define $\mathrm{l}(a)$ as the level at which it is generated, that is, as $\mathrm{min}\{\beta \mid a \in \mathfrak{C}_\beta(\mathcal{A}, \mathbf{I})\}$.

By iterating the cumulative power construction and mirroring the iteration of other power constructions, the above definition is the more natural one to follow. However, as we have noted, the hierarchy is not cumulative in a strong sense, that is, it is not true in general that a level possesses the elements of all preceding levels coexisting. An alternatively defined hierarchy which is cumulative in that sense may be defined in the following way.

\begin{definition}[Strongly cumulative power hierarchy]\label{def hierarchy}
Let $\sigma = \langle \{F_i\}_{i \in I}, \{R_i\}_{i \in J}, \{c_i\}_{i \in K}, \mathrm{ar} \rangle$ be a signature. Let also $\mathbf{I} = \{I_\alpha\}_{\alpha \in \mathbf{On}}$ be a family of sets, and $\mathcal{A} = \langle A, \langle \mathrm{F}_i \rangle_{i \in I}, \langle \mathrm{R}_i \rangle_{i \in J}, \langle \mathrm{c}_i \rangle_{i \in K} \rangle$ be a $\sigma$-structure. The \emph{strongly cumulative power hierarchy of $\mathcal{A}$ generated by $\mathbf{I}$} is composed by the following structures, for each $\beta \in \mathbf{On}$:

\begin{itemize}
\item $\mathfrak{C}_0(\mathcal{A}, \mathbf{I}) = \mathcal{A}$;
\item $\mathfrak{C}_\beta(\mathcal{A}, \mathbf{I}) = \langle \mathcal{F}_\beta(A, \mathbf{I}), \langle \mathrm{F}_{i, \beta} \rangle_{i \in I}, \langle \mathrm{R}_{i, \beta} \rangle_{i \in J}, \langle \mathbf{c}_{i, \beta} \rangle_{i \in K} \rangle$;
\end{itemize}

\noindent 
where $\mathrm{F}_{i, \beta} : (\mathcal{F}_\beta(A, \mathbf{I}))^{\mathrm{ar}(F_i)} \to \mathcal{F}_\beta(A, \mathbf{I})$, $\mathrm{R}_{i, \beta} \subseteq (\mathcal{F}_\beta(A, \mathbf{I}))^{\mathrm{ar}(R_i)}$ and $\mathbf{c}_{i, \beta} \subseteq \mathcal{F}_\beta(A, \mathbf{I})$ are inductively defined in the following way.

\noindent
Now, for $\beta > 0$, we let:

\begin{itemize}

\item[(i)] \emph{The function $\mathrm{F}_{i, \beta}$}: for $\delta = \mathrm{max}\{\rho(a_i)\}_{1 \leq i \leq \mathrm{ar}(\mathrm{F}_i)}$,

\begin{center}
$\mathrm{F}_{i, \beta}(a_1, ..., a_{\mathrm{ar}(\mathrm{F}_i)}) = b$ iff $\delta = \rho(b)$ and $\forall j \in I_{\delta-1} \Big{(} \mathrm{F}_{i, \delta-1}\big{(} \vartheta_j^\delta(a_1), ..., \vartheta_j^\delta(a_{ar(F_i)}) \big{)} = \vartheta^\delta_j(b) \Big{)}$;
\end{center}
\item[(ii)] \emph{The relation $\mathrm{R}_{i, \beta}$}: for $\delta = \mathrm{max}\{\rho(a_i)\}_{1 \leq i \leq \mathrm{ar}(R_i)}$,

$\mathrm{R}_{i, \beta} a_1 ... a_{ar(R_i)}$ iff $\forall j \in I_{\delta-1} \big{(} \mathrm{R}_{i, \delta-1} \vartheta^\delta_j(a_1) ... \vartheta^\delta_j(a_{\mathrm{ar}(\mathrm{R}_i)}) \big{)}$;
\item[(iii)] \emph{The set $\mathbf{c}_{i, \beta}$}:

\begin{itemize}
\item[•] $a \in \mathbf{c}_{i, 1}$ iff $\forall j \in I_0 \big{(} \vartheta^1_j(a) \in \{\mathrm{c}_i\} \big{)}$;
\item[•] $a \in \mathbf{c}_{i, \beta}$ iff $\forall j \in I_{\rho(a)-1} \big{(} \vartheta^{\rho(a)}_j(a) \in \mathbf{c}_{i, \rho(a)-1} \big{)}$.
\end{itemize}
\end{itemize}
\end{definition}

Despite it not being explicit, each level of the hierarchy defined above is strongly cumulative. In fact, it is implicit that, for a limit $\lambda$, $\mathrm{F}_{i, \lambda}$, $\mathrm{R}_{i, \lambda}$, and $\mathbf{c}_{i, \lambda}$ are just the union of the respective sets of the previous levels of the construction, so that it is implicitly defined that, in a strongly cumulative power hierarchy, $\mathfrak{C}_\lambda(\mathcal{A}, \mathbf{I}) = \bigcup_{\alpha < \lambda} \mathfrak{C}_\alpha(\mathcal{A}, \mathbf{I})$. Unfortunately, as we may notice, there is no natural way of defining an interpretation of the constants of a signature in each level of a hierarchy defined in such a way. Instead, to each constant $c_i$ a predicate $\mathbf{c}_{i, \beta}$ is associated.

Since there is a silver lining to each of the hierarchies defined above, both may be of interest -- for example, the strongly cumulative hierarchy may be of especial interest in the case of signatures without constants. As such, we shall study both hierarchies, although, as shall become clear, either one solely suffices for the results of the following sections.

Notice the general scheme of the definition above: for a given atom $\varphi$ and elements $a_0, ..., a_n \in \mathcal{F}_\beta(\mathcal{A}, \mathbf{I})$, $\varphi(a_0, ..., a_n)$ is the case iff for any $j \in I_{\delta-1}$, $\varphi \big{(} \vartheta^\delta_j(a_0), ..., \vartheta^\delta_j(a_n) \big{)}$ holds, where $\delta = \linebreak \max\{\rho(a_i)\}_{0 \leq i \leq n}$.\footnote{Although, in the case of the extended functions, one more constraint is imposed -- and equality, of course, does not accord with that.} Let us call such sort of definition a \emph{hereditary definition}. We shall shortly see what may be said of structures defined hereditarily.

We note that we have used the same notation for the levels of both hierarchies. We hope it does not hinder the readability of the present work, as the context should disambiguate it.

For simplicity, henceforth we shall consider a signature with one $n$-ary function $F$, one $n$-ary relation $R$, and a single constant $c$, whenever that is harmless. It is straightforward how subsequent results apply to larger signatures.

\begin{observation}\label{observation hierarchy}
It is straightforward to check that:

\begin{itemize}[align=parleft, labelsep=8mm,]
\item[(1)] for $\beta \in \mathbf{On}$, the extended structures of the strongly cumulative hierarchy $\mathfrak{C}_\beta(\mathcal{A}, \mathbf{I})$ are well-defined for two reasons: first, by assumption we have $a \in \mathcal{F}_\beta(A, \mathbf{I})$, and therefore, in each case $\delta \leq \beta$, so that $\vartheta^\delta_j(a) \in \bigcup_{\alpha < \beta} \mathcal{F}_\alpha(A, \mathbf{I})$; second, for any $c \in \mathcal{F}_\mathbf{On}(A, \mathbf{I})$, $\rho(c)$ is not a limit ordinal.
\item[(2)] The hereditary function $\vartheta$ allows us to give a single definition for the extended operations and relations for every successor ordinal in the strongly cumulative power hierarchy, but it can be made more explicit. For example, an equivalent definition of a binary function is (for $\delta = \mathrm{max}\{\rho(a), \rho(b)\}$):

\smallskip

\small

\begin{center}
$\mathrm{F}_\beta(a, b) = c$ iff $\begin{cases}
\rho(a) = \rho(b)\ \&\ \forall j \in I_{\delta-1} \Big{(} \mathrm{F}_{\delta-1} \big{(} a(j), b(j) \big{)} = c(j) \Big{)}\\
\rho(b) < \delta\ \&\ \forall j \in I_{\delta-1} \Big{(} \mathrm{F}_{\delta-1} \big{(} a(j), b \big{)} = c(j) \Big{)}\\
\rho(a) < \delta\ \&\ \forall j \in I_{\delta-1} \Big{(} \mathrm{F}_{\delta-1} \big{(} a, b(j) \big{)} = c(j) \Big{)}\\
\mathrm{F}_\delta(a, b) = c,\ \text{if}\ \beta\ \text{is a limit}\\
\end{cases}$
\end{center}

\normalsize

\smallskip

\noindent
Notice that means the definition of $\mathrm{F}_\beta(a, b)$ when $\rho(a) \neq \rho(b)$ is such that $\mathrm{F}_\beta(a, b) = \mathrm{F}_\gamma(a, b)$, for every $\gamma \in [\delta, \beta]$. The intuitive idea in the definitions, represented by the function $\vartheta$, consists in going back to where the largest element is defined and using the operation there, which by assumption is already defined. In the definition of the relations, an equivalent definition would need to break down into the levels of each member of any $n$-tuple. For example, for a binary $\mathrm{R}$, an equivalent definition is:

\smallskip

\begin{center}
$\mathrm{R}_\beta ab$ iff $\begin{cases}
\rho(a) = \rho(b)\ \&\ \forall j \in I_{\delta-1} \big{(} \mathrm{R}_{\delta-1}a(j)b(j) \big{)}\\
\rho(b) < \delta\ \&\ \forall j \in I_{\delta-1} \big{(} \mathrm{R}_{\delta-1} a(j)b \big{)}\\
\rho(a) < \delta\ \&\ \forall j \in I_{\delta-1} \big{(} \mathrm{R}_{\delta-1} a b(j) \big{)}\\
\mathrm{R}_\delta ab,\ \text{if}\ \beta\ \text{is a limit}\\
\end{cases}$
\end{center}

\smallskip

\noindent
We may see the amount of clauses in a definition like that grows proportionally to the number of permutations of an $n$-tuple, so that using the $\vartheta$ function allows for a more economical definition. Likewise, the definition of the set defined by the constants may be easily calculated to be:

\smallskip

\begin{itemize}[align=parleft, labelsep=8mm,]
\item[•] $\mathbf{c}_1 = \mathrm{c}^{I_0} \cup \{\mathrm{c}\}$;\footnote{That is, the collection of all functions from $I_0$ into $\{\mathrm{c}\}$ plus the singleton itself.}
\item[•] $\mathbf{c}_{\beta^+} = \mathbf{c}_\beta^{I_\beta} \cup \mathbf{c}_\beta$;\footnote{Notice $\mathbf{c}_\beta$, for $\beta > 0$, is a set.}
\item[•] $\mathbf{c}_\lambda = \bigcup_{\beta \in \lambda} \mathbf{c}_\beta$.
\end{itemize}

\smallskip

\item[(3)] in the strongly cumulative power hierarchy, for $\alpha < \beta$, $\mathrm{F}_\beta \restriction_{\mathcal{F}_\alpha(A, \mathbf{I})} = \mathrm{F}_\alpha$. Each structure inherits the elements of the domain of all preceding structures, plus all of the functions with domain and codomain in those elements.
\item[(4)] in both hierarchies, for any $\beta \in \mathbf{On}$, $\rho \big{(} \mathrm{F}_\beta(a_0, ..., a_n) \big{)} = \mathrm{max}\{\rho(a_0), ..., \rho(a_n)\}$.
\item[(5)] for any $a \in \mathcal{F}^*_\beta(A, \mathbf{I})$ and $j \in \mathrm{dom}(a)$, $\rho \big{(} a(j) \big{)} < \rho(a)$.
\item[(6)] for any $\beta \in \mathbf{On}$, $j \in \mathbf{I}$ and $a \in \mathcal{F}_\beta(A, \mathbf{I})$, $\rho \big{(} \vartheta^\beta_j(a) \big{)} \leq \rho(a)$.
\item[(7)] in both hierarchies, if $\mathcal{A}$ is closed under $\mathrm{F}$, then $\mathfrak{C}_\beta(\mathcal{A}, \mathbf{I})$ is closed under $\mathrm{F}_\beta$.\footnote{We may see this by induction on $\beta$. The base case is trivial, and for the inductive step it suffices to notice the definition of $\mathrm{F}_\beta$ works iff the $\mathrm{F}_\alpha$ are defined for all $\alpha < \beta$, so that closure is inherited pointwisely. The limit case is straightforward, as two elements in the limit case must be members of some iteration at which the operation is defined for them.}
\end{itemize}
\end{observation}

\begin{proposition}\label{corollary vartheta and functions}
In the strongly cumulative hierarchy, for $a_0, ..., a_n \in \mathcal{F}_\beta(A, \mathbf{I})$, $\delta = max\{\rho(a_i)\}_{i \leq n}$ and $\alpha \leq \delta \leq \beta$:

\begin{itemize}
\item[\emph{(1)}] if $j \in I_\delta$, then $\vartheta^\alpha_j \big{(} \mathrm{F}_\beta(a_0, ..., a_n) \big{)} = \big{(} \mathrm{F}_\beta(a_0, ..., a_n) \big{)}(j) = \mathrm{F}_\beta \big{(} \vartheta^\alpha_j(a_0), ..., \vartheta^\alpha_j(a_n) \big{)}$;
\item[\emph{(2)}] if $j \not\in I_\delta$, then $\vartheta^\delta_j \big{(} \mathrm{F}_\beta(a_0, ..., a_n) \big{)} = \mathrm{F}_\beta(a_0, ..., a_n)$.
\end{itemize}

\noindent
Similarly, in the strongly cumulative hierarchy, for $a_0, ..., a_n \in \mathfrak{C}_\beta(\mathcal{A}, \mathbf{I})$, $\delta = max\{\mathrm{l}(a_i)\}_{i \leq n}$ and $\alpha \leq \delta \leq \beta$:

\begin{itemize}
\item[\emph{(3)}] if $j \in I_\delta$, then $\vartheta^\alpha_j \big{(} \mathrm{F}_\beta(a_0, ..., a_n) \big{)} = \big{(} \mathrm{F}_\beta(a_0, ..., a_n) \big{)}(j) = \mathrm{F}_\beta \big{(} \vartheta^\alpha_j(a_0), ..., \vartheta^\alpha_j(a_n) \big{)}$;
\item[\emph{(4)}] if $j \not\in I_\delta$, then $\vartheta^\delta_j \big{(} \mathrm{F}_\beta(a_0, ..., a_n) \big{)} = \mathrm{F}_\beta(a_0, ..., a_n)$.
\end{itemize}
\end{proposition}

\begin{proof}
Straightforward by Definitions \ref{def hierarchy 1} and \ref{def hierarchy}.\footnote{We offer a quick argument for a binary function: by Observation \ref{observation hierarchy}, $\rho \big{(} \mathrm{F}_\beta(a, b) \big{)} = \delta$. Thus, $\vartheta^\alpha_j \big{(} \mathrm{F}_\beta(a, b) \big{)} = \big{(} \mathrm{F}_\beta(a, b) \big{)} (j)$. Without loss of generality, we may suppose $\delta = \rho(a)$. Then, $\vartheta^\alpha_j(a) = a(j)$. If $\rho(b) < \rho(a)$, then $\vartheta^\alpha_j(b) = b$, and so $\vartheta^\alpha_j \big{(} \mathrm{F}_\beta(a, b) \big{)} = \mathrm{F}_\beta \big{(} a(j), b \big{)} = \mathrm{F}_\beta \big{(} \vartheta^\alpha_j(a), \vartheta^\alpha_j(b) \big{)}$. If now $\rho(a) = \rho(b)$, then $\vartheta^\alpha_j(b) = b(j)$, so that $\vartheta^\alpha_j \big{(} \mathrm{F}_\beta(a, b) \big{)} = \big{(} \mathrm{F}_\beta (a, b) \big{)} (j) = \mathrm{F}_\beta \big{(} a(j), b(j) \big{)} = \mathrm{F}^\beta \big{(} \vartheta^\alpha_j(a), \vartheta^\alpha_j(b) \big{)}$.}
\end{proof}

\begin{corollary}
In both hierarchies, for any $\beta$ and $j \in I_\beta$, $\vartheta_j : \mathfrak{C}_{\beta^+}(\mathcal{A}, \mathbf{I}) \to \mathfrak{C}_\beta(\mathcal{A}, \mathbf{I})$ is an endomorphism (but not an automorphism).
\end{corollary}

\subsection{A concrete example}

In order to make the construction of cumulative powers clearer, we may consider as a concrete case those generated by the natural numbers.

For $\beta < \omega$ and $\{\mathfrak{C}_\alpha(\mathbb{N}, \mathbf{I})\}_{\alpha < \omega}$ itself as the family of index sets, we let

\begin{itemize}
\item[•] $\mathfrak{C}_0(\mathbb{N}, \mathbf{I}) = \mathbb{N}$,
\item[•] $\mathfrak{C}_{\beta^+}(\mathbb{N}, \mathbf{I}) = \mathfrak{C}(\mathfrak{C}_\beta(\mathbb{N}, \mathbf{I}), \mathbf{I}) = \langle \mathcal{F}_\beta(\mathbb{N}, \mathbf{I}), +_\beta, \cdot_\beta, \mathrm{s}_\beta,  \mathbf{0}_\beta \rangle$.
\end{itemize}

Notice there is no circularity in defining the index sets as such, since each $\alpha$th index set, used to construct $\mathfrak{C}_{\alpha^+}(\mathcal{A}, \mathbf{I})$, is taken to be $\mathfrak{C}_\alpha(\mathcal{A}, \mathbf{I})$, and $\mathfrak{C}_0(\mathcal{A}, \mathbf{I})$ is simply $\mathcal{A}$, so that the construction is well-founded.

To illustrate the definition, let us consider a concrete calculation in $\mathfrak{C}_2(\mathbb{N}, \mathbf{I})$, whose domain is

\begin{center}
$\mathcal{F}_2(\mathbb{N}, \mathbf{I}) = \mathcal{F}_1(\mathbb{N}, \mathbf{I})^{\mathcal{F}_1(\mathbb{N}, \mathbf{I})} \cup \mathcal{F}_1(\mathbb{N}, \mathbf{I}) = (\mathbb{N}^\mathbb{N} \cup \mathbb{N})^{\mathbb{N}^\mathbb{N} \cup \mathbb{N}} \cup (\mathbb{N}^\mathbb{N} \cup \mathbb{N})$.
\end{center}

\noindent
For $f \in \mathcal{F}^*_1(\mathbb{N}, \mathbf{I}) = \mathbb{N}^\mathbb{N}$, let $f^+$ be defined as $\forall j \in \mathbb{N} \big{(} f^+(j) = \mathrm{s}(f(j)) \big{)}$. Consider the function $\mathrm{s}_1$, the extension of the successor function $\mathrm{s}$ to $\mathcal{F}_1(\mathbb{N}, \mathbf{I})$:

\begin{center}
$\mathrm{s}_1(j) = 
\begin{cases}
j^+,\ \text{if}\ j \in \mathcal{F}^*_1(\mathbb{N})\\
\mathrm{s}(j) = j + 1,\ \text{if}\ j \in \mathbb{N}
\end{cases}$
\end{center}

\noindent 
In the present construction, notice that $\mathrm{s}_1$ is an element of $\mathcal{F}_2(\mathbb{N}, \mathbf{I})$, since it is a function from $\prescript{\mathbb{N}}{}{\mathbb{N}} \cup \mathbb{N}$ to $\prescript{\mathbb{N}}{}{\mathbb{N}}\cup \mathbb{N}$. Therefore, if we want to calculate the sum of $\mathrm{s}_1$ and $3$, we need to consider the operation $+_2$. The first thing to check is the level of the objects: $\rho(\mathrm{s}_1) = 2$, while $\rho(3) = 0$. Applying Definition \ref{def hierarchy} we have that $\mathrm{s}_1 +_2 3 = a \in \mathcal{F}_2^*(\mathbb{N}, \mathbf{I})$ where $a$ is such that

\smallskip

\begin{center}
$\forall j \in \mathcal{F}_1(\mathbb{N}, \mathbf{I}) \big{(} a(j) = \mathrm{s}_1(j) +_1 3 \big{)}$,
\end{center} 

\smallskip

\noindent 
which means that

\smallskip

\begin{center}
$a(j) = 
\begin{cases}
j^+ +_1 3,\ \text{if}\ j \in \mathbb{N}^\mathbb{N}\\
\mathrm{s}(j) + 3,\ \text{if}\ j \in \mathbb{N}
\end{cases}$
\end{center}

\noindent
Continuing to calculate, we obtain that $j^+ +_1 3 = b \in \mathbb{N}^\mathbb{N}$, where $b$ is such that

\smallskip

\begin{center}
$\forall k \in \mathbb{N} \big{(} b(k) = j^+(k) + 3 \big{)}$.
\end{center}

\smallskip

\noindent
Applying now the definition of $\mathrm{s}_1$ we obtain that, for $k \in \mathbb{N}$, we have that

\smallskip

\begin{center}
$b(k) = j^+(k) + 3 = \mathrm{s} \big{(} j(k) \big{)} + 3$.
\end{center}

\smallskip

In order to have a visual representation of this sum we may take an order of $\mathcal{F}_1(\mathbb{N}, \mathbf{I})$ such that $\mathbb{N}$ is the initial segment, and hence we may represent an element of $\mathcal{F}_2(\mathbb{N}, \mathbf{I})$ by a column vector with the $j$-th coordinate being its value at $j$. Then the first vector below represents $\mathrm{s}_1$, and the sum $\mathrm{s}_1 +_2 3$ can be depicted as follows.

\vspace{2mm}

\small

\begin{center}
$\begin{bmatrix}
1\\
2\\
3\\
\vdots\\
[1, 1, 5, ...]\\
\vdots\\
[1, 3, 7, ...]\\
\vdots
\end{bmatrix} + 3 =$ $
\begin{bmatrix}
1 + 3\\
2 + 3\\
3 + 3\\
\vdots\\
[1, 1, 5, ...] + 3\\
\vdots\\
[1, 3, 7, ...] + 3\\
\vdots
\end{bmatrix} =$ $\begin{bmatrix}
1 + 3\\
2 + 3\\
3 + 3\\
\vdots\\
[1 + 3, 1 + 3, 5 + 3, ...] \\
\vdots\\
[1 + 3, 3 + 3, 7 + 3, ...]\\
\vdots
\end{bmatrix} =$ $\begin{bmatrix}
4\\
5\\
6\\
\vdots\\
[4, 4, 8, ...] \\
\vdots\\
[4, 6, 10, ...]\\
\vdots
\end{bmatrix}
$
\end{center}

\normalsize

\subsection{Preservation and properties of cumulative powers}

If one of the index sets defining a cumulative power hierarchy is either empty or a singleton, in either hierarchy, then the structure generated in its corresponding level is isomorphic to the structure of its preceding level, and therefore uninteresting. Thus, for the rest of the paper, we shall assume any index set to be of cardinality greater than $1$.

The natural question which follows the introduction of cumulative powers amounts to what sentences of the language induced by $\sigma$ are preserved by the construction. For that, some definitions are in order.

As introduced in \cite{horn} (though not called as such there), a \emph{Horn formula} is a formula in prenex conjunctive normal form whose conjuncts are disjunctions of literals of which at most one is a non-negated atom. In other words, a Horn formula is of the form

\begin{center}
$\Delta (\bigwedge_{k \leq m} (\bigvee_{i \leq n_k} \psi_{k, i}) \big{)}$,
\end{center}

\noindent
where each $\psi_{k, i}$ is a literal, and for each $k \leq m$ there is at most one $i \leq n_k$ such that $\psi_{k, i}$ is a non-negated atom. A \emph{Horn sentence} is a closed Horn formula.

In \cite{horn}, Horn showed that any Horn sentence is preserved by direct products. In \cite{changmorel}, Chang and Morel showed Horn sentences are also preserved by reduced products produced by the Fr\'{e}chet filter, and as expressed in \cite{frayne}, Chang further generalized the result, showing that, in fact, any proper (that is, defined by proper filters) reduced product preserves Horn sentences. 

A class of structures defined by a set of sentences is an \emph{elementary class}, and if all those sentences are Horn sentences, call it a \emph{Horn class}. The results mentioned above therefore say that any Horn class is closed by proper reduced (and thus direct) products. Chang and Morel showed in \cite{changmorel} that not every elementary class closed under direct products is a Horn class, but conjectured that any elementary class closed under \emph{reduced} products is indeed a Horn class. In \cite{keisler1965}, Keisler proved the conjecture under the assumption the continuum hypothesis holds for at least one cardinal, whereas in \cite{galvin1970}, Galvin strengthened the result by eliminating that assumption. Briefly:

\begin{proposition}[\cite{changmorel}, \cite{keisler1965}, and \cite{galvin1970}]\label{proposition reduced products}
A sentence is preserved by proper reduced products iff it is equivalent to a Horn sentence.
\end{proposition}

\begin{proposition}[\cite{keisler1965}, and \cite{galvin1970}]\label{proposition reduced products 2}
A sentence is preserved by proper reduced powers iff it is equivalent to a disjunction of Horn sentence.
\end{proposition}

Let the occurrence of an atom in a formula be \emph{positive} if in the prenex conjunctive normal form of that formula the atom is preceded by an even number of negations, and let it be \emph{negative} otherwise. As such, one might suspect the fragment preserved by cumulative powers is the class of Horn sentences without positive equalities, for, as we have seen, there is an imbalance in how the interpretation of constants, functions and other relations are defined in cumulative powers (that is, hereditarily), and how equality is defined (that is, as true identity, rather than hereditary identity). In that case, cumulative powers would preserve strictly fewer first-order sentences than reduced products. However, consider the sentence $\forall x \forall y \exists z (x + y = z)$. As one may check on our previous example of a hierarchy given by the natural numbers, that sentence is preserved and has a positive occurrence of equality. Therefore, the matter is not as simple. Furthermore, as we may comparatively see on the definitions of cumulative powers and direct powers, the former are much closer to the latter than to reduced products or reduced powers. Truth of formulas for functional elements of cumulative powers accord with how their truth is defined in direct powers. Therefore, if we want to characterize the fragment of first-order theory preserved by cumulative powers, our route should be to look at what fragment is preserved by direct powers.

In \cite{weinstein}, Weinstein showed it is possible to give a -- although not simple -- characterization of that fragment not in terms of an inductively defined class of formulas, but as formulas arising from a given recursive decision procedure. To expose it, some definitions are needed.

For $\varphi, \psi, \gamma \in \mathcal{L}^\sigma$, we write $\varphi \times \psi \Rightarrow \gamma$ when, for any $\sigma$-structures $\mathcal{A}$ and $\mathcal{B}$, $a_0, ..., a_n \in A$, and $b_0, ..., b_n \in B$, if $\mathcal{A} \models \varphi(a_0, ...,a_n)$ and $\mathcal{B} \models \psi(b_0, ..., b_n)$, then $\mathcal{A} \times \mathcal{B} \models \gamma(\langle a_0, b_0 \rangle , ..., \langle a_n, b_n \rangle)$. For $\Phi \subseteq \mathcal{L}^\sigma$, we let:

\begin{itemize}
\item $\neg \Phi = \{\neg \varphi \mid \varphi \in \Phi\}$;
\item $\exists x \Phi = \{\exists x \varphi \mid \varphi \in \Phi\}$ (for the particular variable $x$);
\item $\forall x \Phi = \{\forall x \varphi \mid \varphi \in \Phi\}$;
\item $\bigvee \Phi = \{\bigvee_{i \leq n} \varphi_i \mid n < \omega\ \&\ \varphi_i \in \Phi\}$ (including the empty disjunction $\top$);
\item $\bigwedge \Phi = \{\bigwedge_{i \leq n} \varphi_i \mid n < \omega\ \&\ \varphi_i \in \Phi\}$ (including the empty conjunction $\bot$);
\item $\bigwedge^* \Phi = \{\bigwedge_{i  \leq n} \varphi_i \mid n < \omega\ \&\ \varphi_i \in \Phi\ \&\ \bigwedge_{i \leq n} \varphi_i\ \text{is consistent}\}$.
\end{itemize}

In the cases of $\bigvee \Phi$ and $\bigwedge \Phi$, we make the assumption that there is a convention for listing the formulas of $\Phi$ to avoid redundancy -- that is, the members of those sets are as defined modulo logical equivalence. In that way, when $\Phi$ is finite, $\bigvee \Phi$ and $\bigwedge \Phi$ should also be finite, and the mappings $\Phi \mapsto \bigvee \Phi$ and $\Phi \mapsto \bigwedge \Phi$ should be computable. Thus, we may let $\bigwedge^* \Phi$ be such that $\bigwedge^* \Phi \subseteq \bigwedge \Phi$.\footnote{Even though there may not be a computable way to determine $\bigwedge^* \Phi$ as a subset of $\bigwedge \Phi$.}

Let $\langle x_n \rangle_{n < \omega}$ be an enumeration of the variables of $\mathcal{L}^\sigma$. Given a finite set of atomic formulas $\Phi$, we define the following finite sets of atomic formulas:

\begin{itemize}
\item $\Phi_0 = \bigvee \bigwedge (\Phi \cup \neg \Phi)$,
\item $\Phi_{n+1} = \bigvee \bigwedge \exists x_n \Phi_n$, when $n$ is even,
\item $\Phi_{n+1} = \bigvee \bigwedge \forall x_n \Phi_n$, when $n$ is odd,
\item $\Phi^*_0 = \bigvee \bigwedge^* (\Phi \cup \neg \Phi)$,
\item $\Phi^*_{n+1} = \bigvee \bigwedge^* \exists x_n \Phi^*_n$, when $n$ is even,
\item $\Phi^*_{n+1} = \bigvee \bigwedge^* \forall x_n \Phi^*_n$, when $n$ is odd.
\end{itemize}

\noindent
Notice $\Phi^*_n \subseteq \Phi_n$ for every $n < \omega$, and that every formula in $\mathcal{L}^\sigma$ is equivalent to a formula in $\Phi^*_n$ for some $\Phi \subseteq \mathcal{L}^\sigma$ and $n < \omega$, equivalent to a formula in $\Phi_n$ for some $\Phi$ and $n$.

\begin{proposition}[\cite{weinstein}]\label{proposition weinstein}
Let $\Phi$ be a finite autonomous set of formulas closed under $\bigwedge$ and $\bigvee$ up to equivalence.\footnote{The definition for autonomous sets of formulas if offered in \cite{changkeisler}, pp. 425--426. It is, however, not of vital importance for the following results, so we omit it.} Let $x$ be any variable and let $\varphi$, $\varphi'$, and $\psi$ be formulas in $\bigvee \bigwedge \exists x F$. Fix $Q \in \{\exists, \forall\}$. Suppose for every disjunct $\varphi'$ of $\varphi$ and every disjunct $\psi'$ of $\psi$, there is a disjunct $\gamma'$ of $\gamma$ such that for every conjunct $Q x \gamma''$ of $\gamma'$, there are conjuncts $Q x \varphi''$ of $\varphi'$ and $Q x \psi''$ of $\psi'$ such that $\varphi'' \times \psi'' \Rightarrow \gamma''$. Then, $\varphi \times \psi \Rightarrow \gamma$. Moreover, if $\varphi$ and $\psi$ are elements of $\bigvee \bigwedge^* Q x \Phi$, then the supposition is necessary and sufficient for $\varphi \times \psi \Rightarrow \gamma$. 
\end{proposition}

Given the above proposition, Weinstein defines the primitive recursive predicate $R \subseteq (\mathcal{L}^\sigma)^3$ such that $R(\varphi, \psi, \gamma)$ iff there is $\Phi \subseteq \mathcal{L}^\sigma$ and $n < \omega$ such that $\varphi, \psi, \gamma \in \Phi_n$ and:

\begin{itemize}
\item if $n = 0$, then $\varphi \times \psi \Rightarrow \gamma$;
\item if $n$ is odd, then for every disjunct $\varphi'$ of $\varphi$ and every disjunct $\psi'$ of $\psi$, there is a disjunct $\gamma'$ of $\gamma$ such that for every conjunct $\exists x_{n-1} \gamma''$ of $\gamma'$ there are conjuncts $\exists x_{n-1} \varphi''$ of $\varphi'$ and $\exists x_{n-1} \psi''$ of $\psi'$ such that $R(\varphi'', \psi'', \gamma'')$ holds;
\item if $n > 0$ is even, then for every disjunct $\varphi'$ of $\varphi$ and every disjunct $\psi'$ of $\psi$, there is a disjunct $\gamma'$ of $\gamma$ such that for every conjunct $\forall x_{n-1} \gamma''$ of $\gamma'$ there are conjuncts $\forall x_{n-1} \varphi''$ of $\varphi'$ and $\forall x_{n-1} \psi''$ of $\psi'$ such that $R(\varphi'', \psi'', \gamma'')$ holds.
\end{itemize}

Notice that $R$ is a computable predicate, as, given formulas $\varphi$, $\psi$ and $\gamma$, we can take $\Phi$ to be the set of atomic formulas occurring in $\varphi$, $\psi$ and $\gamma$, and it suffices to check for $n$ that is as large as the greatest among the quantifier ranks of those formulas.

Proposition \ref{proposition weinstein} entails the following:

\begin{proposition}[\cite{weinstein}]\label{proposition weinstein 2}
$R(x, y, z)$ satisfies the following:

\begin{itemize}
\item[\emph{(1)}] for any $\varphi$, $\psi$, and $\gamma$ in some $\Phi_n$, if $R(\varphi, \psi, \gamma)$ holds, then $\varphi \times \psi \Rightarrow \gamma$;
\item[\emph{(2)}] for any $\varphi$, $\psi$, and $\gamma$ in some $\Phi^*_n$, if $\varphi \times \psi \Rightarrow \gamma$, then $R(\varphi, \psi, \gamma)$ holds.
\end{itemize}
\end{proposition}

Finally, it was thus proved that:

\begin{corollary}[\cite{weinstein}]\label{corollary weinstein}
A sentence $\varphi \in \mathcal{L}^\sigma$ is preserved by direct products iff it is logically equivalent to a sentence $\psi \in \mathcal{L}^\sigma$ such that, for some $\Phi \subseteq \mathcal{L}^\sigma$ and $n < \omega$, $\psi \in \Phi_n$ and $R(\psi, \psi, \psi)$.
\end{corollary}

Notice that we may replace \emph{products} with \emph{powers} in the above corollary, that is:

\begin{corollary}\label{corollary weinstein 2}
A sentence $\varphi \in \mathcal{L}^\sigma$ is preserved by direct powers iff it is logically equivalent to a sentence $\psi \in \mathcal{L}^\sigma$ such that, for some $\Phi \subseteq \mathcal{L}^\sigma$ and $n < \omega$, $\psi \in \Phi_n$ and $R(\psi, \psi, \psi)$ (iff it is preserved by direct products).
\end{corollary}

\begin{proof}
The left to right direction is straightforward. For the other, if $\varphi$ is preserved by direct powers, then particularly $\varphi \times \varphi \Rightarrow \varphi$. But since, as we have seen, every sentence of $\mathcal{L}^\sigma$ is equivalent to some sentence in $\Phi^*_n$ for some $\Phi \subseteq \mathcal{L}^\sigma$ and $n < \omega$, there is $\psi$ which is logically equivalent to $\varphi$ and $\psi \in \Phi_n$ for some $\Phi \subseteq \mathcal{L}^\sigma$ and $n < \omega$. That also means $\psi \times \psi \Rightarrow \psi$. By Proposition \ref{proposition weinstein 2} (2), we thus obtain $R(\psi, \psi, \psi)$.
\end{proof}

Therefore, Weinstein also proved a characterization of the class of sentences preserved by direct powers. Henceforth, we shall call a sentence which is preserved by direct powers a \emph{direct power sentence}. In the same vein, call a sentence $\psi$ such that, for some $\Phi \subseteq \mathcal{L}^\sigma$ and $n < \omega$, $\psi \in \Phi_n$ and $R(\psi, \psi, \psi)$, a \emph{Weinstein sentence}. Then, Weinstein's result amounts to the fact a sentence is preserved by direct products iff it is a Weinstein sentence. Similarly, the above result says direct power sentences are precisely those equivalent to a Weinstein sentence.

For the upcoming results, we need the following notion. For each index set $I$ and $\sigma$-structure $\mathcal{A}$, we define the mapping $\vartheta^{-1}: \mathfrak{C}(A, I) \to \mathcal{A}^I$ such that, for each $a \in \mathfrak{C}(\mathcal{A}, I)$,

\begin{center}
$\vartheta^{-1}(a) = \begin{cases}
a,\ \text{if}\ a \in \mathfrak{C}(\mathcal{A}, I) \setminus A\\
\overline{a},\ \text{otherwise (that is, if}\ a \in A\text{)}
\end{cases}$
\end{center}

Consider the following result.

\begin{lemma}\label{lemma inverse vartheta}
For any literal $\varphi \in \mathcal{L}^\sigma$ which is not a positive equality and $a_0, ..., a_n \in \mathfrak{C}(\mathcal{A}, I)$,

\begin{center}
$\mathfrak{C}(\mathcal{A}, I) \models \varphi[a_0, ..., a_n]$ iff $\mathcal{A}^I \models \varphi[\vartheta^{-1}(a_0), ..., \vartheta^{-1}(a_n)]$.
\end{center}
\end{lemma}

\begin{proof}
Straightforward from the definitions of cumulative power, direct power, and the $\vartheta^{-1}$ mapping.
\end{proof}

However, unlike other literals -- and as we shall shortly see --, positive equalities are in general not preserved by cumulative powers. To frame those that are (modulo their prenex conjunctive normal form), we need some definitions.

\begin{definition}[Non-collapsible positive equality]\label{def non collapsible}
We call the occurrence of a positive equality $\varphi$ in a formula in prenex conjunctive normal form \emph{non-collapsible} if either:

\begin{itemize}
\item[(i)] $\varphi$ has no variables on one of its sides;
\item[(ii)] or each and every universally quantified or free variable occurring on one side of $\varphi$ also occurs on the other;
\item[(iii)] or each of the sides of $\varphi$ either has variables bound by existential quantifiers, or has the occurrence of a constant.
\end{itemize}

\noindent
If a positive equality in a formula in prenex conjunctive normal form is not non-collapsible, we call it \emph{collapsible}.
\end{definition}

Notice, with the exception of when condition (i) above is satisfied, a positive equality's being non-collapsible is determined by the whole context of the formula in which it occurs (that is, by the quantifiers binding it), rather than just depending on its form alone.

\begin{definition}[Non-collapsible formula]\label{def non collapsible 2}
We call a formula \emph{non-collapsible} if it is in prenex conjunctive normal form $\varphi = \Delta \bigwedge_{i \leq k} \bigvee_{l \leq p_i} \psi_{i, l}$ and either:

\begin{itemize}
\item[(i)] $\varphi$ has no positive occurrences of equality;
\item[(ii)] or each of the positive equalities in $\varphi$ is non-collapsible.
\end{itemize}
\end{definition}

We may loosen our terminology and call an arbitrary formula \emph{non-collapsible} if its prenex conjunctive normal form is non-collapsible. We call a formula as such because they do not force the \emph{collapse} of different levels of the hierarchy, that is, they are formulas whose truth do not require the identification of elements of different levels of the hierarchy. A more detailed explanation of the above definitions is in order. To offer it, we shall consider the context of a simple cumulative power, that is, $\mathfrak{C}(\mathcal{A}, I)$. As one may easily see, a literal which is not a positive equality is preserved in both directions between direct powers and their relative cumulative powers. Thus, the satisfaction of condition (i) of Definition \ref{def non collapsible 2} by a formula entails its preservation. On the other hand, if a sentence satisfies condition (ii) of Definition \ref{def non collapsible 2}, there are three possible cases. If it satisfies condition (iii) of Definition \ref{def non collapsible}, the presence of an existential quantifier on each side of the equality allows us to choose, for whatever choice of elements of $\mathfrak{C}(\mathcal{A}, I)$ on the other side, a suitable element of the same level as those, so that identity may be reached. A constant, on the other hand, does a similar job, since it forces the term in which it occurs to be of the largest possible rank inside the cumulative power structure in question. If it satisfies condition (ii) of Definition \ref{def non collapsible}, since if there is a universally bound variable on one side, then it also occurs on the other side, the ranks of the interpreted terms will always match, for it is possible to choose suitable elements of $\mathfrak{C}(\mathcal{A}, I)$ given any elements picked by the universal quantifiers. The constants, in turn, play a similar role, since in non-strongly cumulative power hierarchies -- those in which they have interpretations -- they are of the highest rank possible in the relevant model, and similarly their values are of the highest possible rank which can be assumed. Thus, if constants are on one side of a positive equality and there are variables bound by existential quantifiers on the other side, we may match the ranks as required.

We may see that cumulative powers do not, in general, preserve formulas containing positive occurrences of equality such that the existential quantifiers occur before the universal quantifiers bounding its variables. In fact, if universal quantifiers bound all of the variables of one side of the equality, then such a sentence is not in general satisfied, because the levels of the chosen elements may not match. Consider for instance $\exists x_1 ... x_n \forall y_1 ... y_m \big{(} t(x_1, ..., x_n) = s(y_1, ..., y_m) \big{)}$, a formula which is collapsible. Since $\rho \big{(} t^{\mathfrak{C}(\mathcal{A}, I)}(x_1, ..., x_n) \big{)} = \mathrm{max}\{\rho(x_i)\}_{1 \leq i \leq n}$ and $\rho \big{(} s^{\mathfrak{C}(\mathcal{A}, I)}(y_1, ..., y_m) \big{)} = \mathrm{max}\{\rho(y_i)\}_{1 \leq i \leq m}$, and since the universal quantifiers may pick elements in either $A$ or $\mathcal{F}^*_1(A, \mathbf{I})$, for a choice $b_1, ..., b_m \in A$ we would have $\rho \big{(} s^{\mathfrak{C}(\mathcal{A}, I)}(b_1, ..., b_m) \big{)} = 0$, which would require $\mathrm{max}\{\rho(x_i)\}_{1 \leq i \leq n} = 0$. However, if $b_1 \in \mathcal{F}^*_1(A, \mathbf{I})$, then $\rho \big{(} s^{\mathfrak{C}(\mathcal{A}, I)}(b_1, ..., b_m) \big{)} = 1$, which would mean $\mathrm{max}\{\rho(x_i)\}_{1 \leq i \leq n} = 1$. Thus, there are no elements $a_1, ..., a_n \in \mathfrak{C}(\mathcal{A}, I)$ satisfying $\forall y_1 ... y_m \big{(} t(a_1, ..., a_n) = s(y_1, ..., y_m) \big{)}$ in $\mathfrak{C}(\mathcal{A}, I)$.

\begin{lemma}\label{lemma non-collapsible preserve}
Let $\varphi \in \mathcal{L}^\sigma$ be in prenex conjunctive normal form. Then:

\begin{itemize}
\item[\emph{(1)}] if $\exists x \varphi$ satisfies condition \emph{(ii)} of \emph{Definition \ref{def non collapsible 2}} because it satisfies either condition \emph{(i)} or \emph{(ii)} of \emph{Definition \ref{def non collapsible}}, or if it satisfies condition \emph{(i)} of \emph{Definition \ref{def non collapsible 2}}, then $\varphi$ is non-collapsible;
\item[\emph{(2)}] if $\forall x \varphi$ is non-collapsible, then $\varphi$ is non-collapsible;
\item[\emph{(3)}] if $\exists x \varphi$ is collapsible, then $\varphi$ is collapsible;
\item[\emph{(4)}] if $\forall x \varphi$ is collapsible, then $\varphi$ is collapsible.
\end{itemize}
\end{lemma}

\begin{proof}
(1) Since $\exists x \varphi$ satisfies condition (i), it has no positive equalities. Taking away a quantifier does not change that fact, and so $\varphi$ still satisfies condition (i).

(2) If condition (i), (ii) or (iii) of Definition \ref{def non collapsible} is satisfied, taking away a universal quantifier does not change their satisfaction. If condition (i) of Definition \ref{def non collapsible 2} is satisfied, the conclusion is also trivial.

(3) and (4) If $\forall x \varphi$ is collapsible, by condition (i), it has positive occurrences of equality. By condition (ii), each side has occurring variables; its sides do not share at least one universally bound or free variable; there are existentially bound variables on at most one of its sides; and the side that does not have existentially bound variables also has no occurrence of constants (for otherwise it would satisfy condition (iii) of Definition \ref{def non collapsible}, and thus condition (ii) of Definition \ref{def non collapsible 2}). Since none of these characteristics are altered by removing a quantifier from $\forall x \varphi$ or $\exists x \varphi$, $\varphi$ is also collapsible.
\end{proof}

\begin{lemma}\label{lemma pga}
Let $a_0, ..., a_n \in \mathfrak{C}(\mathcal{A}, I)$, and $\varphi \in \mathcal{L}^\sigma$ be in prenex disjunctive normal form. Then,

\begin{center}
if $\forall j \in I \big{(} \mathcal{A} \models \varphi \big{[} \vartheta^1_j(a_0), ..., \vartheta^1_j(a_n) \big{]} \big{)}$, then $\mathfrak{C}(\mathcal{A}, I) \models \varphi[a_0, ..., a_n]$.
\end{center}
\end{lemma}

\begin{proof}
We show the contrapositive by an induction on the number of quantifiers. Let $\varphi[a_0, ..., a_n] = \bigvee_{i \leq k} \bigwedge_{l \leq p_i} \psi_{i, l} [a_0, ..., a_n]$ contain no quantifiers. If $\mathfrak{C}(\mathcal{A}, I) \not\models \varphi[a_0, ..., a_n]$, that is, \linebreak $\mathfrak{C}(\mathcal{A}, I) \models \bigwedge_{i \leq k} \bigvee_{l \leq p_i} \neg \psi_{i, l}[a_0, ..., a_n]$, then for each $i' \leq k$ there is $l' \leq p_{i'}$ such that \linebreak $\mathfrak{C}(\mathcal{A}, I) \not\models \psi_{i', l'} [a_0, ..., a_n]$. If $\psi_{i', l'} [a_0, ..., a_n]$ is a negated atom,  we might observe, by our definitions, that means

\begin{center}
$\forall j \in I \big{(} \mathcal{A} \not\models \psi_{i', l'}[\vartheta^\delta_j(a_0), ..., \vartheta^\delta_j(a_n)] \big{)}$.
\end{center}

\noindent
On the other hand, if $\psi_{i', l'} [a_0, ..., a_n]$ is an atom, we may see, by our definitions, that means

\begin{center}
$\exists j \in I \big{(} \mathcal{A} \not\models \psi_{i', l'}[\vartheta^\delta_j(a_0), ..., \vartheta^\delta_j(a_n)] \big{)}$.
\end{center}

\noindent
By the arbitrariness of $l'$ and $i'$, we have

\begin{center}
$\exists j \in I \big{(} \mathcal{A} \not\models \bigwedge_{i \leq k} \bigvee_{l \leq p_i} \psi_{i, l}[\vartheta^\delta_j(a_0), ..., \vartheta^\delta_j(a_n)] \big{)}$.
\end{center}

\noindent
The rest follows by induction on the number of quantifiers.
\end{proof}

\begin{lemma}\label{lemma preservation direct cumulative}
Let $\varphi \in \mathcal{L}^\sigma$ be in prenex conjunctive normal form. If $\varphi$ is non-collapsible, then for any $a_0, ..., a_n \in \mathfrak{C}(\mathcal{A}, I)$,

\begin{center}
if $\mathcal{A}^I \models \varphi[\vartheta^{-1}(a_0), ..., \vartheta^{-1}(a_n)]$, then $\mathfrak{C}(\mathcal{A}, I) \models \varphi[a_0, ..., a_n]$.
\end{center}
\end{lemma}

\begin{proof}
We shall consider separately each condition of Definition \ref{def non collapsible 2}. 

Suppose $\varphi$ satisfies condition (i). We proceed with an induction on the complexity of $\varphi$. If $\varphi$ is atomic, then it cannot be an equality, so the case is covered in both directions by Lemma \ref{lemma inverse vartheta}. For the remaining cases, the non-trivial one is the existential. If $\mathcal{A}^I \models \forall x \varphi[\vartheta^{-1}(a_0), ..., \vartheta^{-1}(a_n)]$, then for any $b \in A^I$, $\mathcal{A}^I \models \varphi[\vartheta^{-1}(a_0), ..., \vartheta^{-1}(a_n), b]$. Since $\vartheta^{-1}$ is surjective, that means for any $c \in \mathfrak{C}(\mathcal{A}, \mathbf{I})$, $\mathcal{A}^I \models \varphi[\vartheta^{-1}(a_0), ..., \vartheta^{-1}(a_n), \vartheta^{-1}(c)]$. Thus, by Lemma \ref{lemma non-collapsible preserve}, $\varphi(x_1, ..., x_n, y)$ is non-collapsible, so by the inductive hypothesis, $\mathfrak{C}(\mathcal{A}, I) \models \varphi[a_0, ..., a_n, c]$. By the arbitrariness of $c$, $\mathfrak{C}(\mathcal{A}, I) \models \forall x \varphi[a_0, ..., a_n]$.

Suppose $\varphi$ satisfies condition (ii) because it satisfies condition (i) of Definition \ref{def non collapsible}. We proceed with an induction on the number of quantifiers. Let $\mathcal{A}^I \models \bigwedge_{i \leq k} \bigvee_{l \leq p_i} \psi_{i, l}[\vartheta^{-1}(a_0), ..., \vartheta^{-1}(a_n)]$ (which is possible if there are constants on each side of the positive equalities). That means for each $i' \leq k$ there is $l' \leq p_{i'}$ such that $\mathcal{A}^I \models \psi_{i', l'}[\vartheta^{-1}(a_0), ..., \vartheta^{-1}(a_n)]$. If $\psi_{i', l'}$ is not a positive equality, then by Lemma \ref{lemma inverse vartheta}, $\mathcal{A}^I \models \psi_{i', l'}[a_0, ..., a_n]$. Otherwise, by assumption it must be of the form $t(x_0, ..., x_n) = s(\mathrm{c}_0, ..., \mathrm{c}_v)$ for some constants $\mathrm{c}_0, ..., \mathrm{c}_v \in \sigma$. Then, $\mathcal{A}^I \models t[\vartheta^{-1}(a_0), ..., \vartheta^{-1}] = s(\mathrm{c}_0, ..., \mathrm{c}_v)$. One may easily check, by our definitions, that implies $\mathfrak{C}(\mathcal{A}, I) \models t[a_0, ..., a_n] = s(\mathrm{c}_0, ..., \mathrm{c}_v)$. By the arbitrariness of $i'$ and $l'$, we therefore have $\mathfrak{C}(\mathcal{A}, I) \models \bigwedge_{i \leq k} \bigvee_{l \leq p_i} \psi_{i, l}[a_0, ..., a_n]$. The inductive step is straightforward by Lemma \ref{lemma non-collapsible preserve} (1) and (2). If $\varphi$ satisfies condition (ii) of Definition \ref{def non collapsible}, as above, the result may be shown by an induction on the number of quantifiers. The base case may be tackled by the same reasoning presented above for when $\varphi$ satisfies condition (i) (that is, going down to the level of literals and applying Lemma \ref{lemma inverse vartheta}), and the argument for the inductive steps are the same. If now $\varphi$ satisfies condition (iii) of Definition \ref{def non collapsible}, we may also deal with the base case and inductive steps in the same way.
\end{proof}

\begin{theorem}\label{theorem preservation direct cumulative}
For any non-collapsible sentence $\varphi \in \mathcal{L}^\sigma$,

\begin{center}
if $\mathcal{A}^I \models \varphi$, then $\mathfrak{C}(\mathcal{A}, I) \models \varphi$.
\end{center}
\end{theorem}

\begin{proof}
Let $\mathcal{A}^I \models \varphi$. If $\varphi$ is a $\Sigma_n$ sentence $\exists x \Delta \psi(x)$, then there is $a \in A^I$ such that \linebreak $\mathcal{A}^I \models \Delta \psi[a]$. Since $\vartheta^{-1}$ is surjective, there is $a' \in \mathfrak{C}(\mathcal{A}, \mathbf{I})$ such that $a = \vartheta^{-1}(a')$. Thus, by Lemma \ref{lemma preservation direct cumulative}, $\mathfrak{C}(\mathcal{A}, I) \models \Delta \psi[a']$, and so $\mathfrak{C}(\mathcal{A}, I) \models \exists x \Delta \psi(x)$. If $\varphi$ is a $\Pi_n$ sentence $\forall x \Delta \psi(x)$, let $a \in \mathfrak{C}(\mathcal{A}, \mathbf{I})$. If $a \in A^I$, then $\mathcal{A}^I \models \Delta \psi[a]$, so by the same reasoning as before, by Lemma \ref{lemma preservation direct cumulative} we get $\mathfrak{C}(\mathcal{A}, I) \models \Delta \psi[a]$, since, in that case, $a = \vartheta^{-1}(a)$. If $a \not\in A^I$, that is, if $a \in A$, notice, by Lemma \ref{lemma non-collapsible preserve} (2), $\Delta \psi(x)$ is also non-collapsible. Thus, by Lemma \ref{lemma preservation direct cumulative}, $\mathfrak{C}(\mathcal{A}, I) \models \Delta \psi[a]$ iff $\mathcal{A}^I \models \Delta \psi[\vartheta^{-1}(a)]$, which we know is the case. Therefore, by the arbitrariness of $a$, $\mathfrak{C}(\mathcal{A}, I) \models \forall x \Delta \psi(x)$.
\end{proof}

\begin{corollary}\label{corollary preservation direct cumulative 2}
Let $\beta \in \mathbf{On}$. In a cumulative power hierarchy, for any non-collapsible sentence $\varphi \in \mathcal{L}^\sigma$,

\begin{center}
if $\mathcal{A}^I \models \varphi$, then $\mathfrak{C}_\beta(\mathcal{A}, I) \models \varphi$.
\end{center}

\noindent
Similarly, in a strongly cumulative power hierarchy, for any constant-free non-collapsible sentence $\varphi \in \mathcal{L}^\sigma$,

\begin{center}
if $\mathcal{A}^I \models \varphi$, then $\mathfrak{C}_\beta(\mathcal{A}, I) \models \varphi$.
\end{center}
\end{corollary}

\begin{proof}
By induction on $\beta$, using Theorem \ref{theorem preservation direct cumulative}.
\end{proof}

For the next result, for a formula $\varphi \in \mathcal{L}^\sigma$, let us write

\begin{center}
$\mathrm{noncoll}(\varphi) = \{\psi \mid \psi\ \text{is a non-collapsible positive equality of}\ \varphi\}$.
\end{center}

\begin{lemma}\label{lemma non collapsible}
Let $\Delta \bigvee_{i \leq k} \psi_i \in \mathcal{L}^\sigma$ not be equivalent to a non-collapsible sentence, where each $\psi_i$ is a literal (that is, $\Delta \bigvee_{i \leq k} \psi_i$ is a formula in prenex normal form with a single conjunct). Then $\varphi$ is not preserved by cumulative powers.
\end{lemma}

\begin{proof}
Since $\Delta \bigvee_{i \leq k} \psi_i$ is collapsible, the conditions of Definition \ref{def non collapsible 2} are falsified. As we have seen in the proof of Lemma \ref{lemma non-collapsible preserve}, $\Delta \bigvee_{i \leq k} \psi_i$ contains a positive equality $t_1 = t_2$ such that: each side has occurring variables; each of its sides do not share at least one universally bound variables; there are existentially bound variables on at most one of its sides; the side which is not existentially bound also has no occurrence of constants. Since $\Delta \bigvee_{i \leq k} \psi_i$ is a sentence, up to renaming variables and constants, we may write

\begin{center}
$\varphi = \Delta \big{(} (\bigvee_{i < k} \psi_i) \vee t_1(x', x_0, ..., x_n) = t_2(x_0, ..., x_n, y_0, ..., y_m, \mathrm{c}_0, ..., \mathrm{c}_r) \big{)}$,
\end{center}

\noindent
where $x_0, ..., x_n, x'$ are universally bound, $y_0, ..., y_m$ are existentially bound, $\mathrm{c}_0, ..., \mathrm{c}_r$ are constants, and no constants occur in $t_1$. We proceed by induction on the number of quantifiers. Now, consider $\Sigma = \big{\{} \neg \gamma \mid \gamma \in \bigvee \big{(} \{\psi_i\}_{i \leq k} \setminus \mathrm{noncoll}(\bigvee_{i \leq k} \psi_i) \big{)} \big{\}} \cup \{t_1 = t_2\}$, and suppose it is not consistent. Then $t_1 = t_2$ logically implies the disjunction of some of those $\gamma$'s. But this would mean that the original equality $t_1 = t_2$ is redundant in the disjunction $\bigvee \big{(} \{\psi_i\}_{i \leq k}$, because whenever the equality holds, one of those $\gamma$'s already holds. Hence the whole disjunction would be logically equivalent to the disjunction obtained by deleting $t_1 = t_2$. Repeating this process for all collapsible equalities would eventually yield a formula that has no collapsible equalities at all -- that is, a non-collapsible sentence -- contradicting the assumption that $\varphi$ is not equivalent to such a sentence. Therefore $\Sigma$ must be consistent. By compactness, that means there is a model $\mathcal{A}$ of $\varphi$ such that (a) $\mathcal{A} \models t_1 = t_2$ and $\mathcal{A} \not\models \gamma$ for every $\gamma \in \bigvee \big{(} \{\psi_i\}_{i \leq k} \setminus \mathrm{noncoll}(\bigvee_{i \leq k} \psi_i) \big{)}$. Suppose

\begin{center}
$\mathfrak{C}(\mathcal{A}, I) \models \forall x \big{(} (\bigvee_{i < k} \psi_i) \vee t_1(x) = t_2(\mathrm{c}_0, ..., \mathrm{c}_r) \big{)} \big{)}$, 
\end{center}

\noindent
Then for any $a \in \mathfrak{C}(\mathcal{A}, I)$, $\mathfrak{C}(\mathcal{A}, I) \models \bigvee_{i < k} \psi_i[a] \vee t_1[a] = t_2(\mathrm{c}_0, ..., \mathrm{c}_r)$. Let $i' < k$. If $\psi_{i'} \in \big{(} \{\psi_i\}_{i \leq k} \setminus \mathrm{noncoll}(\Delta \bigvee_{i \leq k} \psi_i) \big{)}$, then by construction we have $\mathcal{A} \models \neg \psi_{i'}$. By the arbitrariness of $a$, by Lemma \ref{lemma pga}, $\mathfrak{C}(\mathcal{A}, I) \models \neg \psi_{i'}$. Otherwise, $\psi_{i'}$ must be a collapsible positive equality. But then, since there is a single universal quantifier, by the same restrictions exposed at the beginning of our proof, $\psi_{i'}$ must also be of the form $s_1(x) = s_2(\mathrm{c}_0, ..., \mathrm{c}_r)$. Suppose $\mathfrak{C}(\mathcal{A}, I) \models s_1[a] = s_2(\mathrm{c}_0, ..., \mathrm{c}_r)$. By the arbitrariness of $a$, we may let $a \in A$. But then $\rho \big{(} s_1^{\mathfrak{C}(\mathcal{A}, I)}[a] \big{)} = 0$, and $\rho \big{(} s_2^{\mathfrak{C}(\mathcal{A}, I)}(\mathrm{c}_0, ..., \mathrm{c}_r) \big{)} = \rho(\overline{\mathrm{c}_0}) = 1$, which means $\mathfrak{C}(\mathcal{A}, I) \not\models s_1[a] = s_2(\mathrm{c}_0, ..., \mathrm{c}_r)$. Therefore, by the arbitrariness of $i'$, we must have $\mathfrak{C}(\mathcal{A}, I) \models s_1[a] = s_2(\mathrm{c}_0, ..., \mathrm{c}_r)$. But then the same reasoning applies, and we have a contradiction. The inductive steps for the quantifiers are straightforward by Lemma \ref{lemma non-collapsible preserve} (3) and (4).
\end{proof}

\begin{theorem}\label{theorem non collapsible}
Let $\varphi \in \mathcal{L}^\sigma$ not be a equivalent to a non-collapsible sentence. Then $\varphi$ is not preserved by cumulative powers.
\end{theorem}

\begin{proof}
Let $\Delta \bigwedge_{i \leq k} \bigvee_{l \leq p_i} \psi_{i, l}$ be the prenex conjunctive normal form of $\varphi$. If $\varphi$ is preserved, then for each and every $i' \leq k$, $\Delta \bigvee_{l \leq p_{i'}} \psi_{i', l}$ is preserved. But since $\Delta \bigwedge_{i \leq k} \bigvee_{l \leq p_i} \psi_{i, l}$ is not equivalent to a non-collapsible sentence, by Definition \ref{def non collapsible 2} we may see, for some $i''$, $\Delta \bigvee_{l \leq p_{i''}} \psi_{i'', l}$ must equally not be equivalent to a non-collapsible sentence, and therefore, by Lemma \ref{lemma non collapsible}, not preserved by cumulative powers.
\end{proof}

\begin{corollary}\label{corollary characterization cumulative 2}
Let $\varphi \in \mathcal{L}^\sigma$ be a sentence. Then:

\begin{itemize}
\item[\emph{(1)}] $\varphi$ is preserved by cumulative powers iff it is equivalent to a non-collapsible Weinstein sentence; 
\item[\emph{(2)}] $\varphi$ is preserved by the strongly cumulative power hierarchy iff it is equivalent to a constant-free non-collapsible Weinstein sentence.
\end{itemize}
\end{corollary}

\begin{proof}
(1) The left to right direction is given by Theorem \ref{theorem non collapsible}. For the other, suppose $\varphi$ is a non-collapsible Weinstein sentence, then for any $\mathcal{A}$ and index set $I$, if $\mathcal{A} \models \varphi$, then $\mathcal{A}^I \models \varphi$. By Theorem \ref{theorem preservation direct cumulative}, that means $\mathfrak{C}(\mathcal{A}, I) \models \varphi$. The preservation for further levels of the cumulative power hierarchy may be easily given by induction on the levels of the construction.

(2) Similarly, only using Corollary \ref{corollary preservation direct cumulative 2}.
\end{proof}

Therefore, any level of the cumulative power hierarchy inherits any first-order property from $\mathcal{A}$ characterized by non-collapsible Weinstein sentences; and any level of the strongly cumulative power hierarchy inherits any first-order property from $\mathcal{A}$ characterized by constant-free non-collapsible Weinstein sentences -- since in that construction, $\mathfrak{C}_\beta(\mathcal{A}, I)$ does not have an interpretation of the constants of $\sigma$. Not only that, those are precisely the classes of first-order sentences inherited by cumulative powers. In other words, just as the fragment preserved by reduced products is precisely the one given by Horn sentences, and the fragment preserved by direct products is precisely the one characterized by Weinstein, the fragment preserved by (finite) cumulative powers is precisely the one given by non-collapsible Weinstein sentences.

Notice the restrictions of the results on non-collapsible formulas, in comparison to the formulas preserved by direct powers, concern precisely equality. That is related to the definition of the extended functions. In our former example, for $j \in \mathbb{N}$, $\vartheta^1_j(\overline{3}) = \overline{3}(j) = 3$, $\vartheta^1_j(\overline{-1}) = -1$ and $\vartheta^1_j(2) = 2$, so $\forall j \in \mathbb{N} \big{(} \vartheta^1_j(\overline{3}) + \vartheta^1_j(\overline{-1}) = \vartheta^1_j(2) \big{)}$. However, $\mathfrak{C}(\mathbb{N}, \mathbf{I}) \not\models \overline{3} + \overline{-1} = 2$, as $\overline{3} + \overline{-1} = \overline{2} \neq 2$. The restriction has also a deeper meaning. It concerns the difference between the definition of other extended relations and the equality relation. While at each level $\beta$ every other relation is defined by using the $\vartheta$ function, and therefore take into account all of the levels lesser than $\beta$ -- so that, for example, if $\rho(a) < \rho(b) = \beta$ for a successor $\beta$, then $\mathrm{R}_\beta ab$ iff for each $j \in I_{\delta-1}$, $\mathrm{R}_{\beta-1} ab(j)$ --, the equality relation is interpreted as real identity, and thus holds between two objects iff they are the same object, regardless of how they are hereditarily related. In a few words, whereas every other relation is hereditarily defined, equality is not. 

In Section \ref{sec rings}, we shall see that once we also define equality hereditarily, the resulting structure is isomorphic to a direct power.

\begin{table}
\begin{tabular}{ |l|c| }

\hline
\multicolumn{1}{|c|}{Construction} & Fragment of first-order theory preserved \\
\hline
Reduced power $\mathcal{A}^I / F$ & Horn sentences \\ 
\hline
Direct power $\mathcal{A}^I$ & Weinstein sentences \\ 
\hline
Ultrapower $\mathcal{A}^I / \mathcal{U}$ & First-order language \\ 
\hline
Cumulative power $\mathfrak{C}(\mathcal{A}, I)$ & non-collapsible Weinstein sentences \\
\hline
\end{tabular}
\medskip
\caption{Preservation of first-order sentences by power constructions.}
\end{table}

\begin{proposition}
Depending on the choice of $\mathcal{A}$, cumulative powers of $\mathcal{A}$ may produce structures which are not isomorphic to any direct power, reduced power or ultrapower of $\mathcal{A}$.
\end{proposition}

\begin{proof}
Let $\mathrm{T}$ be a theory containing collapsible Horn sentences and $\mathcal{A} \models \mathrm{T}$. By Corollary \ref{corollary characterization cumulative 2}, a cumulative power of $\mathcal{A}$ will not preserve $\mathrm{T}$, whereas $\mathrm{T}$ is preserved by any direct power, reduced power or ultrapower of $\mathcal{A}$.
\end{proof}

So the theory of fields, for example, has cumulative power models which are not isomorphic to direct powers or ultrapowers. Therefore, cumulative powers offer a genuinely new way of constructing algebraic structures, so that theories characterized by non-collapsible Weinstein sentences may find cumulative power models not generated by either direct powers or ultrapowers.

At last, we may see embeddability can be lifted from the generating structures to their respective hierarchies.

\begin{theorem}\label{theorem embeddings 1}
Let $\mathcal{A}$ and $\mathcal{B}$ be $\sigma$-structures, and their (either strongly or non-strongly) cumulative power hierarchies be respectively generated by the families of index sets $\mathbf{I} = \{I_\alpha\}_{\alpha \in \mathbf{On}}$ and $\mathbf{J} = \{J_\alpha\}_{\alpha \in \mathbf{On}}$. If $\mathcal{A} \hookrightarrow \mathcal{B}$ and for each $\alpha \in \mathbf{On}$, $|I_\alpha| \leq |J_\alpha|$, then for any $\beta, \gamma \in \mathbf{On}$ such that $\beta \leq \gamma$, $\mathfrak{C}_\beta(\mathcal{A}, \mathbf{I}) \hookrightarrow \mathfrak{C}_\gamma(\mathcal{B}, \mathbf{J})$.
\end{theorem}

\begin{proof}
Let $e : A \to B$ be the relevant embedding, and for each $\alpha \in \mathbf{On}$ let $u_\alpha : I_\alpha \to J_\alpha$ an injection. We inductively define the mappings $e_\alpha : \mathfrak{C}_\alpha(\mathcal{A}, \mathbf{I}) \to \mathfrak{C}_\alpha(\mathcal{B}, \mathbf{I})$ in the following way:

\begin{itemize}
\item $e_0 = e$;
\item for the inductive step, let $k \in u_\alpha[I_\alpha]$. Then, let $e_{\alpha^+}(a) = b$ iff 

\begin{center}
 $\forall j \in u_\alpha[I_\alpha] \Big{(} b(j) = e_\alpha \big{(} a(u^{-1}_\alpha(j)) \big{)} \Big{)}$ and $\forall j \in J_\alpha \setminus u_\alpha[I_\alpha] \Big{(} b(j) = e_\alpha \big{(} a(u^{-1}_\alpha(k)) \big{)} \Big{)}$.
\end{center}
\end{itemize}

\noindent
For a limit $\lambda$, we naturally let $e_\lambda(a) = b$ iff there is $\delta < \lambda$ such that $e_\delta(a) = b$. One might easily check $e_{\alpha^+}$ is injective, and preserves all the functions, relations and predicates $\mathbf{c}_{\alpha^+}$ of $\sigma$, so that for each $\alpha \in \mathbf{On}$, $e_\alpha$ is an embedding. From there, the conclusion is straightforward.
\end{proof}

Notice the embeddings are not unique, as we may see there are at least $|J_\beta \setminus u_\beta[I_\beta]|$ different embeddings from $\mathfrak{C}_{\beta^+}(\mathcal{A}, \mathbf{I})$ into $\mathfrak{C}_{\beta^+}(\mathcal{A}, \mathbf{J})$.

\section{Direct powers and cumulative powers}\label{sec rings}

As the past results show, both cumulative power hierarchies, as a whole, present a structure that resembles those of their generating structures, but nevertheless fail to instantiate the generating structure itself. How can we recover more of the first-order properties of the generating structures? A way of doing that is by restricting the elements of the hierarchies to those of the parallel hierarchy which is not cumulative -- in other words, a hierarchy of direct powers. Consider the following definition.

\begin{definition}[Direct power hierarchy]\label{def alt hierarchy}
Let $\mathbf{I} = \{I_\alpha\}_{\alpha \in \mathbf{On}}$ be a family of index sets. The \emph{direct power hierarchy generated by $\mathcal{A}$ and $\mathbf{I}$} is composed of the structures $\Pi_\beta(\mathcal{A}, \mathbf{I})$, where each $\Pi_{\beta^+}(\mathcal{A}, \mathbf{I})$ is the direct power of $\Pi_\beta(\mathcal{A}, \mathbf{I})$ with index set $I_\beta$, and for a limit $\lambda$, $\Pi_\lambda(\mathcal{A}, \mathbf{I})$ is the limit of embeddings of the preceding structures. More precisely:

\begin{itemize}
\item $\Pi_0(\mathcal{A}, \mathbf{I}) = \mathcal{A}$;
\item $\Pi_{\beta^+}(\mathcal{A}, \mathbf{I}) = \langle \Pi_\beta(\mathcal{A}, \mathbf{I}), \langle \mathrm{F}_{i, \beta^+} \rangle_{i \in J}, \langle \mathrm{R}_{i, \beta^+} \rangle_{i \in J}, \langle \mathrm{c}_{i, \beta^+} \rangle_{i \in K} \rangle$,\footnote{For simplicity, we use the same symbol to represent the functions, relations and constants of $\mathfrak{C}_\beta(\mathcal{A}, \mathbf{I})$. However, as we shall see, there is no problem in assuming they are the same, as those of $\mathfrak{C}_\beta(\mathcal{A}, \mathbf{I})$ coincide with those of $\Pi_\beta(\mathcal{A}, \mathbf{I})$ when restricted to the domain of the latter.}

\noindent 
where $\mathrm{F}_{i, \beta^+}$, $\mathrm{R}_{i, \beta^+}$ and $\mathrm{c}_{i, \beta^+}$ are inductively defined in the following way:

\begin{itemize}[align=parleft, labelsep=8mm,]

\item[(a)] $\mathrm{F}_{i, \beta^+}(a_0, ..., a_n) = b$ iff $\forall j \in I_\beta \Big{(} \mathrm{F}_{i, \beta} \big{(} a_0(j), ..., a_n(j) \big{)} = b(j) \Big{)}$;
\item[(b)] $\mathrm{R}_{i, \beta^+} a_0...a_n$ iff $\forall j \in I_\beta \big{(} \mathrm{R}_{i, \beta} a_0(j)...a_n(j) \big{)}$;
\item[(c)] $\mathrm{c}_{i, \beta^+} = \overline{\mathrm{c}_{i, \beta}}$;
\end{itemize}

\medskip

\item[•] $\Pi_\lambda(\mathcal{A}, \mathbf{I}) = \bigcup_{\alpha < \lambda} \Pi_\alpha(\mathcal{A}, \mathbf{I}) /\! \equiv_\lambda$, for a limit $\lambda$, is the direct limit of the system $\langle \{\Pi_\alpha(\mathcal{A}, \mathbf{I})\}_{\alpha < \lambda}, e^\alpha_\beta \rangle$ with the embeddings inductively defined by

\medskip

\begin{itemize}
\item $e_\alpha^\alpha$ is the identity function,
\item $e^\alpha_{\beta^+} : \Pi_\alpha(\mathcal{A}, \mathbf{I}) \to \Pi_{\beta^+}(\mathcal{A}, \mathbf{I}); x \mapsto \overline{e^\alpha_{\beta^+}(x)}_\beta$,
\item $e^\alpha_\lambda : \Pi_\alpha(\mathcal{A}, \mathbf{I}) \to \Pi_\lambda(\mathcal{A}, \mathbf{I}); x \mapsto \llbracket x \rrbracket_\lambda$, for a limit $\lambda$,
\end{itemize}
\end{itemize}

\medskip

\noindent
where $\llbracket \cdot \rrbracket_\lambda$ denotes the equivalence class defined by the relation $x \equiv_\lambda y$ iff there are $\alpha \leq \beta < \lambda$ such that either $e^\alpha_\beta(x) = y$ or $e^\alpha_\beta(y) = x$.
\end{definition}

By Corollary \ref{corollary weinstein 2}, we know what sentences are preserved by $\Pi_\beta(\mathcal{A}, \mathbf{I})$. Suppose we let the same family of index sets $\mathbf{I} = \{I_\alpha\}_{\alpha \in \mathbf{On}}$ define both a non-strongly cumulative power and a direct power hierarchy of $\mathcal{A}$. Then, clearly $\Pi_\beta(\mathcal{A}, \mathbf{I}) \hookrightarrow \mathfrak{C}_\beta(\mathcal{A}, \mathbf{I})$, which means there is a fragment of $\mathfrak{C}_\beta(\mathcal{A}, \mathbf{I})$ which is a proper extension of $\mathcal{A}$ that preserves any Weinstein sentence. Suppose, however, we would like to shape the cumulative power hierarchies themselves, with all of its cross-level operations, rather than a fragment of it, into a structure which preserves more first-order properties from its generating structure. That may be done in varying levels of faithfulness, in terms of satisfaction of $\mathcal{L}^\sigma$-sentences. The most faithful option, for successor levels, is to quotient each level $\mathfrak{C}_{\beta^+}(\mathcal{A}, \mathbf{I})$ by an adequate ultrafilter over $I_\beta$. But there is an intermediate level by taking a quotient of $\mathfrak{C}_\beta(\mathcal{A}, \mathbf{I})$ with an appropriate equivalence relation, the product of which is isomorphic to a direct power $\Pi_\beta(\mathcal{A}, \mathbf{I})$.

\begin{definition}\label{def coincidence}
Let $\beta \in \mathbf{On}$. For $a, b \in \mathfrak{C}_\beta(\mathcal{A}, \mathbf{I})$, let $\delta = \mathrm{max}\{\mathrm{l}(a), \mathrm{l}(b)\}$ (analogously,  $\delta = \mathrm{max}\{\rho(a), \rho(b)\}$ in the case of the strongly cumulative hierarchy). We inductively define the equivalence relation $\equiv_\beta$ on $\mathfrak{C}_\beta(\mathcal{A}, \mathbf{I})$ in the following way:

\begin{itemize}[align=parleft, labelsep=8mm,]
\item[(a)] $a \equiv_0 b$ iff $a = b$;
\item[(b)] $a \equiv_{\beta^+} b$ iff $\forall j \in I_{\delta-1} \big{(} \vartheta^\delta_j(a) \equiv_{\delta-1} \vartheta_j^\delta(b) \big{)}$;
\item[(c)] $a \equiv_\lambda b$ iff for some $\alpha < \lambda$, $a \equiv_\alpha b$, for a limit $\lambda$. 
\end{itemize}
\end{definition}

Notice the definition makes sense: $\delta \leq \beta$, which means $\delta-1 < \beta$, and for any $\beta \in \mathbf{On}$, if $\rho(a) = \beta$ (or $\mathrm{l}(a) = \beta$), then by our definition, $\beta$ is not a limit ordinal -- that is, at limit steps of the construction, no new element is introduced --, so that $\delta-1$ indeed exists.

Notice, furthermore, that following our considerations in Section \ref{sec modelling}, the above definition amounts to a hereditary definition of equality, that is, equality is defined modulo the $\vartheta$ function. It should therefore be of no surprise that the resulting structures are isomorphic to direct powers.

\begin{proposition}\label{theorem equiv}
$\equiv_\beta$ is an equivalence relation.
\end{proposition}

\begin{proof}
One might easily check this by induction on $\beta$.
\end{proof}

\begin{proposition}\label{theorem coincidence existence}
For any $\beta^+$ and $a \in \mathfrak{C}_{\beta^+}(\mathcal{A}, \mathbf{I})$ there is $b \in \mathfrak{C}_{\beta^+}(\mathcal{A}, \mathbf{I})$ such that $\mathrm{l}(b) = \beta^+$ (analogously, $\rho(b) = \beta^+$) and $b \equiv_{\beta^+} a$.
\end{proposition}

\begin{proof}
Just consider $\overline{a} \in \mathfrak{C}_{\beta^+}(\mathcal{A}, \mathbf{I})$. Clearly, $\overline{a} \equiv_{\beta^+} a$.
\end{proof}

Thus, when comparing two arbitrary elements with respect to $\equiv_\beta$, we may consider them, without loss of generality, to be of the same level, as, by the above theorem, for any two elements of different levels there is always an element of the highest level between the two to which it is equivalent under $\equiv_\beta$.

Let $a \in \mathfrak{C}_\beta(A, \mathbf{I})$. We define the equivalence classes under $\equiv_\beta$ as usual, and denote that of $a$ by $[a]_\beta$. In the absence of ambiguity, we write simply $[a]$.

\begin{definition}\label{def quotient 1}
Define $\mathcal{C}_\beta(\mathcal{A}, \mathbf{I}) = \langle \mathfrak{C}_\beta(\mathcal{A}, \mathbf{I}) /\! \equiv_\beta, \mathrm{F}^\equiv_\beta, \mathrm{R}^\equiv_\beta, \mathrm{c}^\equiv_\beta \rangle$ as usual -- that is:

\begin{itemize}[align=parleft, labelsep=8mm,]
\item[(a)] $\mathrm{F}^\equiv_\beta : (\mathfrak{C}_\beta(A, \mathbf{I}) /\! \equiv_\beta)^n\ \to (\mathfrak{C}_\beta(A, \mathbf{I}) /\! \equiv_\beta); \langle \llbracket x_0 \rrbracket, ..., \llbracket x_n \rrbracket \rangle \mapsto \llbracket \mathrm{F}_\beta(x_0, ..., x_n) \rrbracket$;
\item[(b)] $\mathrm{c}^\equiv_\beta = \llbracket \mathrm{c} \rrbracket$;
\item[(c)] $\mathrm{R}^\equiv_\beta \llbracket a_0 \rrbracket ... \llbracket a_n \rrbracket$ iff $\exists b_0 , ..., b_n \in \mathfrak{C}_\beta(\mathcal{A}, \mathbf{I}) (a_0 \equiv_\beta b_0\ \&\ ...\ \&\ a_n \equiv_\beta b_n\ \&\ \mathrm{R}_\beta b_0...b_n)$.
\end{itemize}
\end{definition}

We may now easily check that:

\begin{theorem}\label{theorem isomorphism non cum hierarchy}
Let a direct power hierarchy and a (either strongly or non-strongly) cumulative power hierarchy of $\mathcal{A}$ both be generated by the family of index sets $\mathbf{I}$. Then, for any $\beta \in \mathbf{On}$, $\Pi_\beta(\mathcal{A}, \mathbf{I}) \cong \mathcal{C}_\beta(\mathcal{A}, \mathbf{I})$.
\end{theorem}

\begin{proof}
In the case of the non-strongly cumulative hierarchy, we may just take the surjective embedding $e_\beta: \Pi_\beta(\mathcal{A}, \mathbf{I}) \to \mathcal{C}_\beta(\mathcal{A}, \mathbf{I}); x \mapsto [x]_\beta$. In the case of the strongly cumulative hierarchy, notice the previously defined $e$ is also an embedding for $\beta < \omega$. In the case of $\omega$, we may see $e_\omega: \Pi_\omega(\mathcal{A}, \mathbf{I}) \to \mathcal{C}_\omega(\mathcal{A}, \mathbf{I}); \llbracket x \rrbracket_\omega \mapsto [x]_\omega$ is a surjective embedding. For successor levels $\beta^+ \geq \omega$, we recursively define $e_{\beta^+} : \Pi_{\beta^+}(\mathcal{A}, \mathbf{I}) \to \mathcal{C}_{\beta^+}(\mathcal{A}, \mathbf{I}); x \mapsto [e_\beta(x)]_{\beta^+}$, and for a limit $\lambda$, $e_\lambda : \Pi_\lambda(\mathcal{A}, \mathbf{I}) \to \mathcal{C}_\lambda(\mathcal{A}, \mathbf{I}); \llbracket x \rrbracket_\lambda \mapsto [e_\gamma(x)]$, where $\gamma$ is the rank of $x$, that is, the level of the hierarchy of direct powers at which $x$ is generated. It is straightforward to check, by induction on $\beta \geq \omega$, that $e_\beta$ so defined is a surjective embedding.
\end{proof}

Therefore, just as ultrapowers may be seen as direct powers quotiented by a given equivalence relation which disregards the disagreement of values -- identifying elements which disagree in an ultrafilter-small set of indexes --, direct powers may be seen as cumulative powers quotiented by an equivalence relation which disregards real identity -- identifying elements which agree up to hereditary identity. Thus, cumulative powers are, arguably, a generalisation of direct powers.

\section{Ultrapowers and cumulative powers}

As we have previously mentioned, to shape the hierarchy into the same first-order structure of its generating structure, it suffices to adapt the ultrapower construction to the present case. That would entail taking a quotient of each level $\mathfrak{C}_{\beta^+}(\mathcal{A}, \mathbf{I})$ by an ultrafilter over $I_\beta$. The problem, of course, is that there are many elements in $\mathfrak{C}_{\beta^+}(A, \mathbf{I})$ which are not functions. More so, for a successor $\beta^+$, there are pairs of elements that are functions, but whose domains differ. So simply taking a quotient of $\mathfrak{C}_{\beta^+}(\mathcal{A}, \mathbf{I})$ by an ultrafilter over $I_\beta$ is not a valid construction. To remedy that, we shall define equivalence classes modulo an ultrafilter by using equality modulo the $\vartheta$ function, as in the construction of $\mathcal{C}_\beta(\mathcal{A}, \mathbf{I})$.

\begin{definition}\label{def sim relation}
Let $a, b \in \mathfrak{C}_\beta(\mathcal{A}, \mathbf{I})$. Consider a family $\{\mathcal{U}_\alpha\}_{\alpha \in \mathbf{On}}$ of ultrafilters, where each $\mathcal{U}_\beta$ is an ultrafilter over $I_\beta$. We inductively define a hierarchy of relations such that:

\begin{itemize}
\item[(a)] $a \sim_0 b$ iff $a = b$,
\end{itemize}

\noindent
For $\beta \geq 0$ and $\delta = \mathrm{max}\{\mathrm{l}(a), \mathrm{l}(b)\}$ (or $\delta = \mathrm{max}\{\rho(a), \rho(b)\}$),

\begin{itemize}
\item[(b)] $a \sim_{\beta^+} b$ iff $ \{j \in I_\beta \mid \vartheta_j^\delta(a) \sim_\beta \vartheta^\delta_j(b)\} \in \mathcal{U}_\beta $
\end{itemize}

\noindent
For a limit $\lambda$, just let $a \sim_\lambda b$ iff $a \sim_\alpha b$ for some $\alpha < \lambda$.\footnote{Notice, in the case of the strongly cumulative power hierarchy, $\sim_\lambda\ = \bigcup_{\alpha < \lambda} \sim_\alpha$.}
\end{definition}

\medskip

Thus, intuitively, $a \sim_{\beta^+} b$ when either $a \sim_\alpha b$ for $\alpha < \beta^+$, or when $a$ and $b$ differ modulo $\sim_\beta$ in an amount of points that are ultrafilter-small. Now, we may see that:

\begin{proposition}\label{corollary concidence sim}
For either cumulative power hierarchy and its associated relations $\equiv_\beta$ and $\sim_\beta$, for any $\beta \in \mathbf{On}$, $\equiv_\beta\ \subseteq\ \sim_\beta$.
\end{proposition}

\begin{proof}
To see so, just notice the resemblance between the inductive definition of both relations.
\end{proof}

Therefore, similarly to our previous consideration of Proposition \ref{theorem coincidence existence}, when proving results about $\sim_\beta$ we may consider different elements of the cumulative power hierarchy to be of the same level without loss of generality. More precisely, whenever we have $a \sim_{\beta^+} b$ for some arbitrary $a$ and $b$, we may let $\rho(a) = \rho(b) = \beta^+$ (respectively, $\mathrm{l}(a) = \mathrm{l}(b) = \beta^+$), which then means (for $\delta = max\{\rho(a), \rho(b)\}$, or $\delta = max\{\mathrm(a), \mathrm(b)\}$, and $j \in I_\beta$) $\vartheta^\delta_j(a) = a(j)$ and $\vartheta^\delta_j(b) = b(j)$, so that

\medskip

\begin{center}
$a \sim_{\beta^+} b$ iff $\{j \in I_\beta \mid a(j) \sim_\beta b(j)\} \in \mathcal{U}_\beta$.
\end{center}

\medskip

\begin{proposition}\label{theorem sim equivalence}
$\sim_\beta$ is an equivalence relation on $\mathfrak{C}_\beta(\mathcal{A}, \mathbf{I})$.
\end{proposition}

\begin{proof}
By induction on $\beta$. The base and limit cases are straightforward, as is the inductive step with respect to reflexivity and symmetry. For transitivity, let $a \sim_{\alpha^+} b \sim_{\alpha^+} c$. As observed above, we may let $\rho(a) = \rho(b) = \rho(c) = \mathrm{max}\{\rho(a), \rho(b), \rho(c)\} = \alpha^+$ (analogously, $\mathrm{l}(a) = \mathrm{l}(b) = \mathrm{l}(c) = \mathrm{max}\{\mathrm{l}(a), \mathrm{l}(b), \mathrm{l}(c)\} = \alpha^+$). By Definition \ref{def sim relation}, that means

\begin{center}
$\{j \in I_\alpha \mid a(j) \sim_\alpha b(j)\}, \{j \in I_\alpha \mid b(j) \sim_\alpha c(j)\} \in \mathcal{U}_\alpha$,
\end{center}

\noindent
and therefore

\begin{center}
$\{j \in I_\alpha \mid a(j) \sim_\alpha b(j)\} \cap \{j \in I_\alpha \mid b(j) \sim_\alpha c(j)\} \in \mathcal{U}_\alpha$.
\end{center}

\noindent
But by induction hypothesis,

\begin{center}
$\{j \in I_\alpha \mid a(j) \sim_\alpha b(j)\} \cap \{j \in I_\alpha \mid b(j) \sim_\alpha c(j)\} \subseteq \{j \in I_\alpha \mid a(j) \sim_\alpha c(j)\}$,
\end{center}

\noindent
and therefore $a \sim_{\alpha^+} c$.
\end{proof}

Let $a \in \mathfrak{C}_\beta(\mathcal{A}, \mathbf{I})$. We define its equivalence class with respect to $\sim_\beta$ as usual, and denote it by $[a]_\beta$. We may omit the subscripts, writing $[a]$ and $\sim$, when the level of the respective construction is clear by context.

\begin{definition}\label{def quotient 2}
We define the structures $\mathbb{C}_\beta(\mathcal{A}, \mathbf{I}) = \langle \mathfrak{C}_\beta(\mathcal{A}, \mathbf{I}) /\! \sim_\beta, \mathrm{F}_\beta^\sim, \mathrm{R}_\beta^\sim, \mathrm{c}^\sim_\beta \rangle$ in the usual manner:

\begin{itemize}[align=parleft, labelsep=8mm,]
\item[(a)] $\mathrm{F}^\sim_\beta : (\mathfrak{C}_\beta(\mathcal{A}, \mathbf{I}) /\! \sim_\beta)^n \to (\mathfrak{C}_\beta(\mathcal{A}, \mathbf{I}) /\! \sim_\beta); \langle [x_0], ..., [x_n] \rangle \mapsto [\mathrm{F}_\beta(x_0, ..., x_n)]$;
\item[(b)] $\mathrm{c}^\sim_\beta = [\mathrm{c}]$;
\item[(c)] $\mathrm{R}^\sim_\beta [a_0]...[a_n]$ iff $\exists b_0 , ..., b_n \in \mathfrak{C}_\beta(\mathcal{A}, \mathbf{I}) (a_0 \sim_\beta b_0\ \&\ ...\ \&\ a_n \sim_\beta b_n\ \&\ \mathrm{R}_\beta b_0...b_n)$.
\end{itemize}
\end{definition}

\begin{lemma}\label{lemma relation quotient}
Let $a_0, ..., a_n, b \in \mathfrak{C}_{\beta^+}(\mathcal{A}, \mathbf{I})$, and $\mathcal{U}_\beta$ be the ultrafilter used for the definition of $\sim_{\beta^+}$. Then:

\begin{itemize}
\item[\emph{(a)}] \small $\mathrm{F}^\sim_{\beta^+}([a_0], ..., [a_n]) = [b]$ iff $\{j \in I_\beta \mid \mathrm{F}^\sim_\beta \big{(} [a_0(j)], ..., [a_n(j)] \big{)} = [b(j)]\} \in \mathcal{U}_\beta$\normalsize ;
\item[\emph{(b)}] $\mathrm{R}^\sim_{\beta^+} [a_0]...[a_n]$ iff $\{j \in 
I_\beta \mid \mathrm{R}^\sim_\beta [a_0(j)]...[a_n(j)]\} \in \mathcal{U}_\beta$.
\end{itemize} 
\end{lemma}

\begin{proof}
(a): By definition, $\mathrm{F}^\sim_{\beta^+}([a_0], ..., [a_n]) = [b]$ iff $b \sim_{\beta^+} \mathrm{F}^\sim_{\beta^+}(a_0, ..., a_n)$, iff

\begin{center}
$\{j \in I_\beta \mid \mathrm{F}_\beta \big{(} a_0(j), ..., a_n(j) \big{)} \sim_\beta b(j)\} \in \mathcal{U}_\beta$,\footnote{Notice, in the context of our previous observations, we are under the assumption $\rho(a_i) = \rho(b_i) = \beta^+$.}
\end{center}

\noindent
but

\medskip

\begin{center}
$\{j \in I_\beta \mid \mathrm{F}_\beta \big{(} a_0(j), ..., a_n(j) \big{)} \sim_\beta b(j)\} =$\\ 
$\{j \in I_\beta \mid \big{[} \mathrm{F}_\beta \big{(} a_0(j), ..., a_n(j) \big{)} \big{]} = [b(j)]\} =$\\
$\{j \in I_\beta \mid \mathrm{F}^\sim_\beta \big{(} [a_0(j)], ..., [a_n(j)] \big{)} = [b(j)]\}$.
\end{center}

\medskip

(b): By definition, $\mathrm{R}^\sim_{\beta^+}[a_0] ... [a_n]$ iff there are $b_0, ..., b_n \in \mathfrak{C}_{\beta^+}(\mathcal{A}, \mathbf{I})$ such that $a_i \sim_{\beta^+} b_i$ and $\mathrm{R}_{\beta^+}b_0 ... b_n$. By Definitions \ref{def hierarchy 1} and \ref{def hierarchy}, that is the case iff

\begin{center}
$\{j \in I_\beta \mid \mathrm{R}_\beta b_0(j) ... b_n(j)\} = I_\beta \in \mathcal{U}_\beta$,
\end{center}

\noindent
and

\begin{center}
$\{j \in I_\beta \mid a_i(j) \sim_\beta b_i(j) \} \in \mathcal{U}_\beta$.
\end{center}

\noindent
Since $n < \omega$, that means the intersection of all of the above sets is in $\mathcal{U}_\beta$, so by definition, that is the case iff (since an ultrafilter is closed by containment)

\begin{center}
$\{j \in I_\beta \mid \mathrm{R}^\sim_\beta [a_0(j)] ... [a_n(j)] \} \in \mathcal{U}_\beta$.
\end{center}
\end{proof}

Then, it is straightforward to adapt \L o\'{s}'s theorem to the present case, so that we know:

\begin{theorem}\label{theorem los}
For any $\varphi \in \mathcal{L}^\sigma$ and $a_1, ..., a_n \in \mathfrak{C}_{\beta^+}(\mathcal{A}, \mathbf{I})$,

\begin{center}
$\mathbb{C}_{\beta^+}(\mathcal{A}, \mathbf{I}) \models \varphi \big{[} [a_1], ..., [a_n] \big{]}$ iff $\{j \in I_\beta \mid \mathbb{C}_\beta(\mathcal{A}, \mathbf{I}) \models \varphi \big{[} [a_1(j)], ..., [a_n(j)] \big{]} \} \in \mathcal{U}_\beta$.
\end{center}
\end{theorem}

\begin{proof}
Straightforward by Definition \ref{def quotient 2} and Lemma \ref{lemma relation quotient}, and by induction on the complexity of $\varphi$.
\end{proof}

We thus obtain:

\begin{corollary}\label{corollary los}
For any $\alpha \leq \beta$, $\mathbb{C}_\alpha(\mathcal{A}, \mathbf{I}) \preceq \mathbb{C}_\beta(\mathcal{A}, \mathbf{I})$, and therefore $\mathcal{A} \preceq \mathbb{C}_\beta(\mathcal{A}, \mathbf{I})$.
\end{corollary}

Therefore, taking a quotient of any stage of the hierarchy by the $\sim$ relation produces a structure elementarily equivalent to the one generating the hierarchy.

As we have noted, the present construction holds many similarities to ultrapowers. Indeed, just as we have shown $\mathcal{C}_\beta(\mathcal{A}, \mathbf{I})$ is isomorphic to a direct power, we shall now show $\mathbb{C}_\beta(\mathcal{A}, \mathbf{I})$ is isomorphic to an ultrapower. For that, we start with some definitions.

\begin{definition}[Ultrapower hierarchy]\label{def ultrapower hier}
Let $\mathcal{A}$ be a $\sigma$-structure, $\mathbf{I} = \{I_\alpha\}_{\alpha \in \mathbf{On}}$ a family of index sets, and $\{\mathcal{U}_\alpha\}_{\alpha \in \mathbf{On}}$ a family of ultrafilters, where each $\mathcal{U}_\alpha$ is an ultrafilter over $I_\alpha$. The \emph{ultrapower hierarchy} generated by those parameters is defined as (for each $\alpha^+, \lambda \in \mathbf{On}$):

\begin{itemize}[align=parleft, labelsep=8mm,]
\item[•] $\Upsilon_0(\mathcal{A}, \mathbf{I}) = \mathcal{A}$;
\item[•] $\Upsilon_{\alpha^+}(\mathcal{A}, \mathbf{I}) = \Upsilon_\alpha(\mathcal{A}, \mathbf{I})^{I_\alpha} / \mathcal{U}_\alpha$;
\item[•] $\Upsilon_\lambda(\mathcal{A}, \mathbf{I}) = \bigcup_{\alpha < \lambda} \Upsilon_\alpha(\mathcal{A}, \mathbf{I}) /\! \approx_\lambda$, for a limit $\lambda$, is the direct limit of the system $\langle \{\Upsilon_\alpha(\mathcal{A}, \mathbf{I})\}_{\alpha < \lambda}, e^\alpha_\beta \rangle$, with the embeddings inductively defined by

\medskip

\begin{itemize}
\item $e^\alpha_\alpha$ is the identity function,
\item $e^\alpha_{\beta^+} : \Upsilon_\alpha(\mathcal{A}, \mathbf{I}) \to \Upsilon_{\beta^+}(\mathcal{A}, \mathbf{I}); x \mapsto \llbracket \overline{e^\alpha_\beta(x)} \rrbracket_{\mathcal{U}_\beta}$,
\item $e^\alpha_\lambda : \Upsilon_\alpha(\mathcal{A}, \mathbf{I}) \to \Upsilon_\lambda(\mathcal{A}, \mathbf{I}); x \mapsto \llbracket x \rrbracket_\lambda$, for a limit $\lambda$,
\end{itemize}

\medskip

\noindent
where $\llbracket \cdot \rrbracket_{\mathcal{U}_\beta}$ denotes the equivalence class modulo $\mathcal{U}_\beta$ for $\beta^+$, and $\llbracket \cdot \rrbracket_\lambda$, for a limit $\lambda$, the equivalence class defined by the relation $\approx_\lambda$, given by $x \approx_\lambda y$ iff there are $\alpha \leq \beta < \lambda$ such that either $e^\alpha_\beta(x) = y$ or $e^\alpha_\beta(y) = x$. When the ordinal is unambiguous, we may write simply $\llbracket \cdot \rrbracket$.
\end{itemize}

\noindent
For $a \in \bigcup_{\alpha \leq \beta} \Upsilon_\alpha(\mathcal{A}, \mathbf{I})$, let $\mathrm{l}(a) = \mathrm{min}\{\alpha \mid \exists x \in \Upsilon_\alpha(\mathcal{A}) (x \approx a)\}$. We might recall the interpretation of the non-logical symbols in an ultrapower is given by (for $\delta = \mathrm{max}\{\mathrm{l}(a_1), ..., \mathrm{l}(a_n)\} = \linebreak min\{\alpha < \lambda \mid \forall x \in B (x \cap \Upsilon_\alpha(\mathcal{A}) \neq \varnothing)\}$):

\begin{itemize}[align=parleft, labelsep=8mm,]
\item[•] $\mathrm{F}_i^{\Upsilon_{\beta^+}(\mathcal{A}, \mathbf{I})}$ maps $\langle \llbracket a_0 \rrbracket, ..., \llbracket a_n \rrbracket \rangle \mapsto \llbracket f \rrbracket$, where $f(x) = \mathrm{F}_i^{\Upsilon_\beta(\mathcal{A}, \mathbf{I})} \big{(} a_0(x), ..., a_n(x) \big{)}$ for any $x \in \Upsilon_\beta(\mathcal{A}, \mathbf{I})$, and for a limit $\lambda$, $\mathrm{F}_i^{\Upsilon_\lambda(\mathcal{A}, \mathbf{I})}$ maps $\langle \llbracket a_0 \rrbracket, ..., \llbracket a_n \rrbracket \rangle \mapsto \big{\llbracket} \mathrm{F}_i^{\Upsilon_\delta(\mathcal{A}, \mathbf{I})} \big{(} e^{\rho(a_0)}_\delta(a_0), ..., e^{\rho(a_n)}_\delta(a_n) \big{)} \big{\rrbracket}$;
\item[•] $\mathrm{R}_j^{\Upsilon_{\beta^+}(\mathcal{A}, \mathbf{I})} \llbracket a_0 \rrbracket ... \llbracket a_n \rrbracket$ iff $\{x \in \Upsilon_\beta(\mathcal{A}, \mathbf{I}) \mid \mathrm{R}_j^{\Upsilon_\beta(\mathcal{A}, \mathbf{I})} a_0(x) ... a_n(x)\} \in \mathcal{U}_\beta$, and for a limit $\lambda$, \linebreak $\mathrm{R}_j^{\Upsilon_\lambda(\mathcal{A}, \mathbf{I})} \llbracket a_0 \rrbracket ... \llbracket a_n \rrbracket$ iff $\mathrm{R}_j^{\Upsilon_\delta(\mathcal{A}, \mathbf{I})} e^{\rho(a_0)}_\delta(a_0) ... e^{\rho(a_n)}_\delta(a_n)$;
\item[•] $\mathrm{c}_k^{\Upsilon_{\beta^+}(\mathcal{A}, \mathbf{I})} = \llbracket \overline{\mathrm{c}_k^{\Upsilon_\beta(\mathcal{A}, \mathbf{I})}} \rrbracket$, and for a limit $\lambda$, $\mathrm{c}_k^{\Upsilon_\lambda(\mathcal{A}, \mathbf{I})} = \llbracket \mathrm{c}^\mathcal{A} \rrbracket$.
\end{itemize}
\end{definition}

By replacing above $\beta$ for $\mathbf{On}$, we define the whole ultrapower hierarchy and the class $\Upsilon_\mathbf{On}(\mathcal{A}, \mathbf{I})$.\footnote{Notice $\Upsilon_\mathbf{On}(\mathcal{A}, \mathbf{I})$ will not be a proper class only if there is a cardinal equal to the number of ultrafilters $\mathcal{U}_\gamma$ in $\{\mathcal{U}_\alpha\}_{\alpha \in \mathbf{On}}$ such that $\mathcal{U}_\gamma$ is non-principal.} For any $\llbracket a \rrbracket \in \Upsilon_\mathbf{On}(\mathcal{A}, \mathbf{I})$, we may say $\llbracket a \rrbracket \in \Upsilon_\beta(\mathcal{A}, \mathbf{I})$ when $\llbracket a \rrbracket \cap \Upsilon_\beta(\mathcal{A}, \mathbf{I}) \neq \varnothing$. Similarly, for $X \subseteq \Upsilon_\mathbf{On}(\mathcal{A}, \mathbf{I})$, we may say $X \subseteq \Upsilon_\beta(\mathcal{A}, \mathbf{I})$ when for each $\llbracket a \rrbracket \in X$, $\llbracket a \rrbracket \cap \Upsilon_\beta(\mathcal{A}, \mathbf{I}) \neq \varnothing$.

We might note this is a stepwise, and sometimes direct, construction of iterated ultrapowers, and in fact when considering the finite stages of the construction, coincides with finitely iterated ultrapowers.\footnote{Developed for finite iterations by \cite{frayne}, and for infinite iterations by \cite{gaifman} (\cite{changkeisler}, Chapter 6).} Notice, however, that despite the use of proper classes for the definition of the ultrapower hierarchy as a whole, the present construction is somewhat less reliant on proper classes. Whereas in the analogous iterated ultrapower construction of the proper class sized $\mathcal{A}^{I_\mathbf{On}} / \mathcal{U}_\mathbf{On}$, where $I_\mathbf{On} = \bigtimes_{\alpha < \mathbf{On}} I_\alpha$ and $\mathcal{U}_\mathbf{On} = \bigtimes_{\alpha < \mathbf{On}} \mathcal{U}_\alpha$, any statement about its elements involves the consideration of its proper equivalence classes, any statement about a set sized $X \subseteq \Upsilon_\mathbf{On}(\mathcal{A}, \mathbf{I})$ is equivalent to a statement about $\Upsilon_\delta(\mathcal{A}, \mathbf{I})$, where $\delta = \mathrm{min}\{\alpha \in \mathbf{On} \mid \forall x \in X (x \cap \Upsilon_\alpha(\mathcal{A}, \mathbf{I}) \neq \varnothing)\}$ -- that is, the first stage at which each element of $X$ has a representative.

\begin{lemma}\label{lemma equivalent constructions}
Let $\lambda$ be a limit ordinal. If for all $\beta < \lambda$, $\mathbb{C}_\beta(\mathcal{A}, \mathbf{I}) \cong \Upsilon_\beta(\mathcal{A}, \mathbf{J})$, then $\mathbb{C}_\lambda(\mathcal{A}, \mathbf{I}) \cong \Upsilon_\lambda(\mathcal{A}, \mathbf{J})$.
\end{lemma}

\begin{proof}
Let $\theta_\beta : \mathbb{C}_\beta(\mathcal{A}, \mathbf{I}) \to \mathfrak{C}_\beta(A, \mathbf{I})$ be a choice function on the equivalence classes of $\mathbb{C}_\beta(\mathcal{A}, \mathbf{I})$, for each $\beta < \lambda$. In the case of either hierarchies, consider $e: \Upsilon_\lambda(\mathcal{A}, \mathbf{I}) \to \mathbb{C}_\lambda(\mathcal{A}, \mathbf{I}); \llbracket x \rrbracket_{\mathcal{U}_\lambda} \mapsto [\theta_\delta \big{(} e_\delta(\llbracket x \rrbracket_{\mathcal{U}_\delta}) \big{)}]_\lambda$, where $\delta$ is the stage of the ultrapower hierarchy at which $x$ is generated, that is, $x \in \Upsilon_{\delta-1}(\mathcal{A}, \mathbf{I})^{I_{\delta-1}}$, and $e_\delta$ is the isomorphism between $\Upsilon_\delta(\mathcal{A}, \mathbf{I})$ and $\mathbb{C}_\delta(\mathcal{A}, \mathbf{I})$. 
\end{proof}

\begin{theorem}\label{theorem equivalent constructions 1}
Let a cumulative power hierarchy be generated by $\mathcal{A}$ and the family of index sets $\mathbf{I} = \{I_\alpha\}_{\alpha \in \mathbf{On}}$. Let an ultrapower hierarchy similarly be generated by the same structure and family of index sets, and the family of ultrafilters $\{\mathcal{U}_\alpha\}_{\alpha \in \mathbf{On}}$ respectively over them. For each $\beta \in \mathbf{On}$, if for all $\beta < \gamma$, $\sim_{\beta^+}$ is defined by $\mathcal{U}_\beta$, then $\mathbb{C}_\gamma(\mathcal{A}, \mathbf{I}) \cong \Upsilon_\gamma(\mathcal{A}, \mathbf{I})$.
\end{theorem}

\begin{proof}
By induction on $\gamma$. The base case is trivial, and the limit case is covered by Lemma \ref{lemma equivalent constructions}. For the successor case, let $e_\alpha: \Upsilon_\alpha(\mathcal{A}, \mathbf{I}) \to \mathbb{C}_\alpha(\mathcal{A}, \mathbf{I})$ be the isomorphism. Let $\theta: \mathbb{C}_\alpha(\mathcal{A}, \mathbf{I}) \to \mathfrak{C}_\alpha(\mathcal{A}, \mathbf{I})$ be a choice function on the equivalence classes. Consider $w: \Upsilon_\alpha(\mathcal{A}, \mathbf{I})^{I_\alpha} \to \mathbb{C}_{\alpha^+}(\mathcal{A}, \mathbf{I})$ such that for each $f \in \Upsilon_\alpha(\mathcal{A}, \mathbf{I})^{I_\alpha}$, $w(f): x \mapsto \theta \big{(} e_\alpha(f(x)) \big{)}$. We now show that $e_{\alpha^+}: \Upsilon_{\alpha^+}(\mathcal{A}, \mathbf{I}) \to \mathbb{C}_{\alpha^+}(\mathcal{A}, \mathbf{I}); \llbracket x \rrbracket_{\mathcal{U}_\alpha} \mapsto [w(x)]_{\alpha^+}$ is an isomorphism. It is trivial to see that $e_{\alpha^+}$ preserves the interpretations of constants, that is, that for any constant $\mathrm{c}$, $e_{\alpha^+}(\mathrm{c}^{\Upsilon_{\alpha^+}(\mathcal{A}, \mathbf{I})}) = e_\alpha^+(\llbracket \overline{\mathrm{c}^{\Upsilon_\alpha(\mathcal{A}, \mathbf{I})}} \rrbracket_{\mathcal{U}_\alpha}) = [w(\overline{\mathrm{c}^{\Upsilon_\alpha(\mathcal{A}, \mathbf{I})}})]_{\alpha^+} = [\overline{\mathrm{c}^{\mathbb{C}_\alpha(\mathcal{A}, \mathbf{I})}}]_{\alpha^+} = \mathrm{c}^{\mathbb{C}_{\alpha^+}(\mathcal{A}, \mathbf{I})}$. Let $e_\alpha^+(\llbracket f \rrbracket_{\mathcal{U}_\alpha}) = [w(f)]_{\alpha^+} = [w(g)]_{\alpha^+} = e_{\alpha^+}(\llbracket g \rrbracket_{\mathcal{U}_\alpha})$, meaning $\{j \in I_\alpha \mid w(f)(j) \sim_\alpha w(g)(j)\} \in \mathcal{U}_\alpha$. We have $w(f)(j) \sim_\alpha w(g)(j)\}$ iff $\theta \big{(} e_\alpha(f(j)) \big{)} \sim_\alpha \theta \big{(} e_\alpha(g(j)) \big{)}$, so since $\theta$ is a fixed choice function on the equivalence classes, that is the case iff $e_\alpha(f(j)) = e_\alpha(g(j))$. Since, by assumption, $e_\alpha$ is injective, that is the case iff $f(j) = g(j)$. Therefore, $\{j \in I_\alpha \mid f(j) = g(j)\} \in \mathcal{U}_\alpha$, and thus $\llbracket f \rrbracket_{\mathcal{U}_\alpha} = \llbracket g \rrbracket_{\mathcal{U}_\alpha}$. Let now $\mathrm{R}^{\Upsilon_{\alpha^+}(\mathcal{A}, \mathbf{I})} \llbracket f_0 \rrbracket_{\mathcal{U}_\alpha} ... \llbracket f_n \rrbracket_{\mathcal{U}_\alpha}$. That is the case iff $\{j \in I_\alpha \mid \mathrm{R}^{\Upsilon_\alpha(\mathcal{A}, \mathbf{I})} f_0(j)...f_n(j)\} \in \mathcal{U}_\alpha$. Since $e_\alpha$ is an isomorphism, $\mathrm{R}^{\Upsilon_\alpha(\mathcal{A}, \mathbf{I})} f_0(j)...f_n(j)$ iff $\mathrm{R}^{\mathbb{C}_\alpha(\mathcal{A}, \mathbf{I})} e_\alpha(f_0(j)) ... e_\alpha(f_n(j))$, iff, since $\theta$ is a choice function on the equivalence classes, \linebreak $\mathrm{R}^{\mathbb{C}_\alpha(\mathcal{A}, \mathbf{I})} [\theta \big{(} e_\alpha(f_0(j)) \big{)}]_\alpha ... [\theta \big{(} e_\alpha(f_n(j)) \big{)}]_\alpha$, iff  $\mathrm{R}^{\mathbb{C}_\alpha(\mathcal{A}, \mathbf{I})} [w(f_0)(j)]_\alpha ... [w(f_n)(j)]_\alpha$. Thus, $\{j \in I_\alpha \mid \mathrm{R}^{\mathbb{C}_\alpha(\mathcal{A}, \mathbf{I})} [w(f_0)(j)]_\alpha ... [w(f_n)(j)]_\alpha\} \in \mathcal{U}_\alpha$, so by Lemma \ref{lemma relation quotient}, $\mathrm{R}^{\mathbb{C}_{\alpha^+}(\mathcal{A}, \mathbf{I})} [w(f_0)]_{\alpha^+} ... [w(f_n)]_{\alpha^+}$. By a similar argument, we may show $e_{\alpha^+}$ preserves the functions. At last, to see $e_{\alpha^+}$ is surjective, it suffices to show for any $f \in \mathfrak{C}_{\alpha^+}(\mathcal{A}, \mathbf{I})$ there is $f' \in \Upsilon_\alpha(\mathcal{A}, \mathbf{I})^{I_\alpha}$ such that $f \sim_{\alpha^+} w(f')$. To see that, consider $f'$ such that for any $j \in I_\alpha$, $f'(j) = e^{-1}_\alpha \big{(} [f(j)]_\alpha \big{)}$. Then, for any $j \in I_\alpha$ $w(f')(j) = \theta \big{(} e_\alpha(f'(j)) \big{)} =  \theta \big{(} e_\alpha \big{(} e^{-1}_\alpha \big{(} [f(j)]_\alpha \big{)} \big{)} \big{)} = \theta \big{(} [f(j)]_\alpha \big{)} \sim_\alpha f(j)$, which means $\{j \in I_\alpha \mid w(f')(j) \sim_\alpha f(j)\} \in \mathcal{U}_\alpha$, so that $w(f') \sim_{\alpha^+} f$.
\end{proof}

Therefore, just as it was the case with direct powers, an ultrapower may be seen to be a structure emerging from taking a quotient of a cumulative power by an appropriately defined equivalence relation.

\section{Real closed fields}\label{sec hyp sur}

In this section, we shall apply the hierarchical setting to show a new, and natural, construction of the surreal field $\mathbf{No}$ as the whole hierarchy of either adequate ultrapowers, or, equivalently, quotiented cumulative power -- that is, as a direct limit of those constructions.

For an ordered set $\mathcal{A}$ and $f, g \in A^A$, say $f$ \emph{dominates} $g$ if there is $j \in A$ such that for all $k > j$, $f(k) > g(k)$. Given an order $\langle A, < \rangle$ and $L, R \subseteq A$, we write $L < R$, as usual, to denote that every element of $L$ is less than every element of $R$. Given $a \in A$, for simplicity we write $a < L$ for $\{a\} < L$. Call $B \subseteq A^A$ \emph{unbounded} if there is no $h \in A^A$ such that $h$ dominates every element of $B$. Following \cite{vandouwen}, define $\mathfrak{b}_A$ to be the size of the smallest unbounded subset of $A^A$. \cite{hechler1975} has shown that $\aleph_0 < \mathfrak{b}_\omega$. We adapt here that proof to the more general case:

\begin{lemma}\label{lemma jump cardinality 1}
Let $B \subseteq A^A$ be unbounded and $\{B_\tau\}_{\tau < \mathrm{cof}(\mathcal{A})}$ be a decomposition of $B$. Then, for some $\tau$, $B_\tau$ is unbounded.
\end{lemma}

\begin{proof}
Suppose the latter is false, so that for any $\kappa < \mathrm{cof}(\mathcal{A})$ there is $f_\kappa \in A^A$ such that for all $g \in B_\kappa$, $f_\kappa$ dominates $g$. Since each $\kappa < \mathrm{cof}(\mathcal{A})$, there is $l_\kappa \in A$ such that $l_\kappa > \{f_\tau(\kappa)\}_{\tau \leq \kappa}$. Consider $f: A \to A; x \mapsto l_x$. Then $f$ dominates every $f_\kappa$ for $\kappa < \mathrm{cof}(\mathcal{A})$, and therefore every element of $A^A$. That means $A^A$ is bounded.
\end{proof}

Since $A^A$ is unbounded iff $\mathcal{A}$ is unbounded, we obtain:

\begin{corollary}\label{corollary jump cardinality 1}
If $\mathcal{A}$ is unbounded, then $\mathrm{cof}(\mathcal{A}) < \mathfrak{b}_\mathcal{A}$.
\end{corollary}

\begin{definition}[Tails filter]
A filter $F$ over an ordered set $\langle I, \leq \rangle$ is a \emph{tails filter} if it contains every tail of the order, that is, if for every $k \in I$, $\{j \in I \mid k \leq j\} \in F$. 
\end{definition}

By the definition of $\mathfrak{b}_\mathcal{A}$, we may immediately see that:

\begin{lemma}\label{lemma jump cardinality 2}
If $\mathcal{A}$ is unbounded and $\mathcal{U}$ is a tails ultrafilter over $A$, then $\mathfrak{b}_\mathcal{A} \leq \mathrm{cof}(\mathcal{A}^A / \mathcal{U})$.
\end{lemma}

By the above lemma and Corollary \ref{corollary jump cardinality 1}, we have:

\begin{corollary}\label{corollary jump cardinality 2}
If $\mathcal{A}$ is unbounded and $\mathcal{U}$ is a tails ultrafilter over $A$, then $\mathrm{cof}(\mathcal{A}) < \mathrm{cof}(\mathcal{A}^A / \mathcal{U})$. Therefore, if also $|A| = \mathrm{cof}(\mathcal{A})$, then $|A| < |\mathcal{A}^A / \mathcal{U}|$.
\end{corollary}

For $X \subseteq A$, let $[X]$ denote the image of the natural embedding of $X$ into the ultrapower $\mathcal{A}^I / \mathcal{U}$. Say $\mathcal{A}^I / \mathcal{U}$ \emph{dominates} $[X]$ if there is $a \in A^I$ such that $[a] > [\overline{x}]$ for all $x \in X$ -- that is, if there is an upper bound of $[X]$ in $\mathcal{A}^I / \mathcal{U}$.

\begin{lemma}\label{lemma tails ultrafilter order}
Let $\mathcal{A}$ be a totally ordered set. Then $\mathcal{A}^I / \mathcal{U}$ dominates $[A]$ iff there is a total order $\langle I, <' \rangle$ with the same cofinality as $\mathcal{A}$ over which $\mathcal{U}$ is a tails ultrafilter.
\end{lemma}

\begin{proof}
($\Leftarrow$) Just let $X \subseteq A$ be cofinal with $A$ and consider an order preserving $f : I \to X$. Then $[\overline{n}] < [f]$ for any $n \in A$. 

($\Rightarrow$) Let $f \in A^X$ with $[f] > [\overline{x}]$ for all $x \in A$. Equivalently, for all such $x$, $\{j \in A \mid x < f(j)\} \in \mathcal{U}$. Then, it is straightforward to see $I$ must be of the same cardinality as the cofinality of $\mathcal{A}$. Notice if $m, n \in A$ and $m < n$, then

\begin{center}
$\{j \in A \mid n < f(j)\} \subset \{j \in A \mid m < f(j)\}$.
\end{center}

\noindent
Then, $\big{\{} \{j \in A \mid n < f(j)\} \big{\}}_{n \in A}$ can be linearly ordered by supersets. Define

\begin{center}
$l <' k$ iff\\ 
$\big{\{} n \in A \mid k \in \{j \in A \mid n < f(j)\} \big{\}} \subset \big{\{} n \in A \mid l \in \{j \in A \mid n < f(j)\} \big{\}}$
\end{center}

\noindent
Then $\langle I, <' \rangle$ is a total order with the same cofinality as $\mathcal{A}$, and one might easily check $\mathcal{U}$ is a tails ultrafilter over it.
\end{proof}

\begin{definition}[Saturation]
A structure $\mathcal{A}$ is \emph{$\kappa$-saturated} if for any $X \subseteq A$ with $|X| < \kappa$, $\mathcal{A}$ realizes any complete type over $X$.\footnote{A type over $X \subseteq A$ is a set $\Phi$ of formulas with $n$ free variables and parameters in $X$ such that any finite subset of $\Phi$ is satisfiable in $\mathcal{A}$. A type is realized by $\mathcal{A}$ if there is an $n$-tuple in $A^n$ satisfying $\Phi$, and complete if for any $\varphi$ with parameters in $X$, either $\varphi \in \Phi$ or $\neg \varphi \in \Phi$.}
\end{definition}

Due to Hausdorff \cite{hausdorff}, we also define the following:

\begin{definition}[$\eta_\beta$-set]
For $\beta \in \mathbf{On}$, an ordered set $A$ is an \emph{$\eta_\beta$-set} if for any $L, R \subseteq A$ with $|L|, |R| < \aleph_\beta$ and $L < R$ there is $a \in A$ such that $L < a < R$. Likewise, $A$ is an \emph{$\eta_\mathbf{On}$-class} if it is a proper class that is a \emph{$\eta_\beta$-class} for every $\beta \in \mathbf{On}$.\footnote{That is, such that for any sets $X, Y \subset A$ such that $X < Y$ there is $a \in A$ such that $X < a < Y$.}
\end{definition}

\begin{proposition}[\cite{changkeisler}, p. 343]\label{proposition saturation eta}
An unbounded densely ordered set is a \emph{$\eta_\beta$-set} iff it is $\aleph_\beta$-saturated.
\end{proposition}

It is now straightforward to check that:

\begin{lemma}\label{lemma saturation 1}
If $\mathcal{A}$ is an $\eta_0$-set, then $\mathcal{A}^I / \mathcal{U}$ is an $\eta_0$-set.
\end{lemma}

\begin{proof}
By Proposition \ref{proposition saturation eta} and the properties of ultrapowers.
\end{proof}

\begin{lemma}\label{lemma saturation 2}
Let $\mathcal{A}$ be an $\eta_0$-set. If $\mathcal{U}$ is a tails ultrafilter over a total order $\langle I, <' \rangle$ with the same cofinality as $\mathcal{A}$, then for every $L, R \subseteq A$ such that $L < R$ there is $c \in \mathcal{A}^I / \mathcal{U}$ such that $[L] < c < [R]$.
\end{lemma}

\begin{proof}
Without loss of generality, we may assume $(L, R)$ is either a gap or a cut, for otherwise by Lemma \ref{lemma saturation 1} we may take an element between the last element of $L$ and the first element of $R$. Suppose first the former. Let $\kappa_L$ and $\kappa_R$ be the cofinality and coinitiality of $L$ and $R$, respectively. Let first $\kappa_L \geq \kappa_R$. Then there is an order reversing mapping $v : L \to R$ with an image coinitial in its codomain. By the definition of $L$ and $R$, for any $l \in L$, $l < v(l)$, so by Lemma \ref{lemma saturation 1} $\big{(} l, v(l) \big{)}$ is a non-empty interval. That means we may take a choice function $\theta$ on the set of open intervals $\bigcup_{l \in L} \big{\{} \big{(} l, v(l) \big{)} \big{\}}$. Let $f : I \to \theta \big{(} \bigcup_{l \in L} \big{\{} \big{(} l, v(l) \big{)} \big{\}} \big{)}$ with an image cofinal in its codomain and such that if $a <' b$ then $f(a) \leq f(b)$. Let now $c \in L$. Then, we have

\begin{center}
$[\overline{c}] <_\mathcal{U} [f]$ iff $\{j \in I \mid c < f(j)\} \in \mathcal{U}$.
\end{center}

\noindent
Since the image of $f$ is cofinal in its codomain, there is $d \in I$ such that $f(d) \geq \theta \big{(} \big{(} c, v(c) \big{)} \big{)}$. Consider $k \in I$ such that $k >' d$. Then, by our definitions, $f(k) \geq f(d) \geq \theta \big{(} \big{(} c, v(c) \big{)} \big{)} > c$. Since $\mathcal{U}$ is a tails ultrafilter on $\langle I, <' \rangle$, by the arbitrariness of $k$ we conclude $[\overline{c}] <_\mathcal{U} [f]$. By the arbitrariness of $c$, we may see $[L] <_\mathcal{U} [f]$. By similar considerations, we may also see $[f] <_\mathcal{U} [R]$. Let now $\kappa_R \geq \kappa_L$. We proceed analogously: we take an order preserving mapping $v : R \to L$ with an image cofinal in its codomain. Similarly to the previous case, for any $r \in R$, $\big{(} v(r), r \big{)}$ is a non-empty interval, so we may take a choice function $\theta$ on $\bigcup_{r \in R} \big{\{} \big{(} v(r), r \big{)} \big{\}}$, and let $f : I \to \theta \big{(} \bigcup_{r \in R} \big{\{} \big{(} v(r), r \big{)} \big{\}} \big{)}$ with an image cofinal in its codomain such that if $a <' b$ then $f(a) \leq f(b)$. We then proceed with the same argument as the previous one, and arrive to the same conclusion. Suppose now $(L, R)$ is a cut. Then either $L$ has a last element or $R$ has a first element. Suppose the former, and let $M \in L$ be the last element. Similarly to the previous case, for any $r \in R$, by Lemma \ref{lemma saturation 1} $(M, r)$ is non-empty, so we take a choice function $\theta$ on the set $\bigcup_{r \in R} \{ (M, r) \}$ and let $f : I \to \theta \big{(} \bigcup_{r \in R} \big{\{} \big{(} M, r \big{)} \big{\}} \big{)}$ have an image cofinal in its codomain, and be such that if $a <' b$ then $f(a) \leq f(b)$. By the same argument as before, we obtain $[L] <_\mathcal{U} [f] <_\mathcal{U} [R]$. If, on the other hand, $R$ has a first element $m$, an analogous argument may be employed by taking a choice function on $\bigcup_{l \in L} \{ (l, m) \}$. 
\end{proof}

\begin{theorem}\label{theorem initial fragment}
Let $\lambda$ be a limit ordinal, $\mathcal{A}$ be an $\eta_0$-set, and its ultrapower hierarchy be generated by taking $\mathbf{I} = \{\Upsilon_\alpha(\mathcal{A}, \mathbf{I})\}_{\alpha \in \mathbf{On}}$ as the index sets. Let $\aleph_\lambda$ be strongly inaccessible.\footnote{As we recall, despite the appearance from the notation, defining the family of index set as such is non-circular.} For each $\alpha < \lambda$, suppose there are $\beta, \gamma < \lambda$ such that $|\Upsilon_\beta(\mathcal{A}, \mathbf{I})| > |\Upsilon_\alpha(\mathcal{A}, \mathbf{I})|$ and $\Upsilon_\gamma(\mathcal{A}, \mathbf{I})$ dominates $[\Upsilon_\alpha(\mathcal{A}, \mathbf{I})]$. Then, $\Upsilon_\lambda(\mathcal{A}, \mathbf{I})$ is an $\eta_\lambda$-set (equivalently, $\aleph_\lambda$-saturated).
\end{theorem}

\begin{proof}
By the first supposition, $\aleph_\lambda \leq |\Upsilon_\lambda(\mathcal{A}, \mathbf{I})|$. Let $L, R \subseteq A_\lambda$ with $|L|, |R| < \aleph_\lambda$ and $L <_\lambda R$. Since $\aleph_\lambda$ is strongly inaccessible, there is some $\beta < \lambda$ such that $L, R \subseteq \Upsilon_\beta(\mathcal{A}, \mathbf{I})$. By the second supposition and Lemma \ref{lemma tails ultrafilter order}, there is $\gamma > \beta$ such that $\mathcal{U}_\gamma$ is a tails ultrafilter over the totally ordered set $\langle \Upsilon_\gamma(\mathcal{A}, \mathbf{I}), <' \rangle$ with the same cofinality as $\Upsilon_\gamma(\mathcal{A}, \mathbf{I})$. By Lemma \ref{lemma saturation 2}, there is $c \in \Upsilon_{\gamma^+}(\mathcal{A}, \mathbf{I})$ such that $L <_{\gamma^+} c <_{\gamma^+} R$. Once again, $\lambda$ is a limit, so that $\gamma^+ < \lambda$, and therefore $L <_\lambda c <_\lambda R$.
\end{proof}

In other words, starting with an $\eta_0$-set $\mathcal{A}$, if there are enough increases in cardinality throughout the ultrapower hierarchy up to a limit $\lambda$, and if one may always saturate the order of a given level at a subsequent level that is lower than $\lambda$, \emph{and} given $\aleph_\lambda$ is regular, then $\Upsilon_\lambda(\mathcal{A}, \mathbf{I})$ is saturated in every cardinality lower than $\aleph_\lambda$. Notice, of course, if $|A| = \aleph_\lambda$, then the increases in cardinality are needless, and the theorem follows without that assumption. We therefore might see that:

\begin{theorem}\label{theorem initial fragment 2}
Let $\mathcal{A}$ be an $\eta_0$-set, and for $\beta < \lambda$, let each successor stage $\beta^+$ of an ultrapower hierarchy be constructed by taking the predecessor level $\Upsilon_\beta(\mathcal{A}, \mathbf{I})$ as the index set and a tails ultrafilter over it. If $\aleph_\lambda$ is strongly inaccessible, then, $\Upsilon_\lambda(\mathcal{A}, \mathbf{I})$ is an $\eta_\lambda$-set.
\end{theorem}

\begin{proof}
Since each $\Upsilon_{\beta^+}(\mathcal{A}, \mathbf{I}) = \Upsilon_\beta(\mathcal{A}, \mathbf{I})^{\Upsilon_\beta(\mathcal{A}, \mathbf{I})} / \mathcal{U}_\beta$, where $\mathcal{U}_\beta$ is a tails ultrafilter over $\Upsilon_\beta(\mathcal{A}, \mathbf{I})$, by Corollary \ref{corollary jump cardinality 2}, $\mathrm{cof} \big{(} \Upsilon_\beta(\mathcal{A}, \mathbf{I}) \big{)} < \mathrm{cof} \big{(} \Upsilon_{\beta^+}(\mathcal{A}, \mathbf{I}) \big{)}$, and so $|\Upsilon_\lambda(\mathcal{A}, \mathbf{I})| \geq \mathrm{cof}\big{(} \Upsilon_\lambda(\mathcal{A}, \mathbf{I}) \big{)} = \mathrm{sup}\{\mathrm{cof}\big{(} \Upsilon_\alpha(\mathcal{A}, \mathbf{I}) \big{)} \mid \alpha < \lambda\} \geq \aleph_\lambda$. If we now let $L, R \subseteq \Upsilon_\lambda(\mathcal{A})$ with $L <_\lambda R$ and $|L|, |R| < \aleph_\lambda$, by Lemma \ref{lemma saturation 2} (using the same argument we have for Theorem \ref{theorem initial fragment}), we may see $\Upsilon_\lambda(\mathcal{A}, \mathbf{I})$ is an $\eta_\lambda$-set.
\end{proof}

Notice particularly that if $|A| = \aleph_0$, under the assumption of GCH, $|\Upsilon_\lambda(\mathcal{A}, \mathbf{I})| = \aleph_\lambda$, and is therefore unique up to order isomorphism (\cite{hausdorff}, pp. 180--185).\footnote{One might see, since we have GCH, that by induction, $|A_0| = |A| = \aleph_0  < \mathrm{cof}(A_1) \leq |A^A| = 2^{\aleph_0} = \aleph_1$.}

The above results imply:

\begin{corollary}\label{corollary initial fragment}
Let $\mathcal{A}$ be an $\eta_0$-set and each successor level $\Upsilon_{\beta^+}(\mathcal{A}, \mathbf{I})$ be constructed by taking the predecessor $\Upsilon_\beta(\mathcal{A}, \mathbf{I})$ as the index set (that is, $\mathbf{I} = \{\Upsilon_\alpha(\mathcal{A}, \mathbf{I})\}_{\alpha \in \mathbf{On}}$) and a tails ultrafilter over it. Then, $\Upsilon_\mathbf{On}(\mathcal{A}, \mathbf{I})$ is an $\eta_\mathbf{On}$-class (equivalently, saturated in every cardinal).
\end{corollary}

Given Theorem \ref{theorem equivalent constructions 1}, the above results imply that if $\mathcal{A}$ is an $\eta_0$-set, $\{\Upsilon_\alpha(\mathcal{A}, \mathbf{I})\}_{\alpha \in \mathbf{On}}$ is the family of index sets generating the strongly cumulative power hierarchy, and each ultrafilter $\{\mathcal{U}_\alpha\}_{\alpha \in \mathbf{On}}$ defining $\sim_{\alpha^+}$ is a tails ultrafilter, then for each strongly inaccessible $\aleph_\lambda$, $\mathbb{C}_\lambda(\mathcal{A}, \mathbf{I})$ is an $\eta_\lambda$-set, and $\mathbb{C}_\mathbf{On}(\mathcal{A}, \mathbf{I})$ is an $\eta_\mathbf{On}$-class. However, we may eliminate the dependence on ultrapowers of that sort of construction by showing that, for any ultrapower hierarchy constructed by taking each successor level as an ultrapower having the preceding level as its index set -- as constructed above --, there is a corresponding hierarchy of cumulative powers constructed by taking each preceding level as the index set of its successor such that, once quotiented by the equivalence relations $\sim$ defined by appropriate choices of ultrafilter, its structures are isomorphic to their ultrapower counterparts.

\begin{lemma}\label{lemma existence covarying filter}
Let $\equiv$ be an equivalence relation on a set $A$. Then every ultrafilter $\mathcal{U}$ over $A /\! \equiv$ induces a filter base $F$ over $A$ such that $\{x \in A \mid [x] \in X\} \in F$ iff $X \in \mathcal{U}$.
\end{lemma}

\begin{proof}
Let $F' = \big{\{} \{x \in A \mid [x] \in X\} \mid X \in \mathcal{U} \big{\}}$. Since $A /\! \equiv\ \in \mathcal{U}$ and $\varnothing \not\in \mathcal{U}$, we have $A \in F$ and $\varnothing \not\in F$. Let now $X, Y \in F$. Then, $\{x \in A \mid [x] \in X\}, \{x \in A \mid [x] \in Y\} \in \mathcal{U}$, which means $\{x \in A \mid [x] \in X\} \cap \{x \in A \mid [x] \in Y\} = \{x \in A \mid [x] \in X \cap Y\} \in \mathcal{U}$. Thus, $X \cap Y \in F$. That $F$ has the described property follows from its definition and the fact $\mathcal{U}$ is an ultrafilter.
\end{proof}

\begin{theorem}\label{theorem equivalent constructions 2}
Let an ultrapower hierarchy be generated by $\mathcal{A}$ and the families of index sets $\{\Upsilon_\alpha(\mathcal{A}, \mathbf{I})\}_{\alpha \in \mathbf{On}}$ and ultrafilters $\{\mathcal{U}_\alpha\}_{\alpha \in \mathbf{On}}$ respectively over them. Then, there is a family of ultrafilters $\{\mathcal{U}'_\alpha\}_{\alpha \in \mathbf{On}}$ respectively over $\{\mathfrak{C}_\alpha(A, \mathbf{I})\}_{\alpha \in \mathbf{On}}$ such that for each $\gamma \in \mathbf{On}$, if for all $\delta < \gamma$, $\sim_{\delta^+}$ is defined by $\mathcal{U}'_\delta$, then $\mathbb{C}_\gamma(\mathcal{A}, \mathbf{I}) \cong \Upsilon_\gamma(\mathcal{A}, \mathbf{I})$.
\end{theorem}

\begin{proof}
The structure of the proof is similar to that of Theorem \ref{theorem equivalent constructions 1}, by induction on $\gamma$. The base case is trivial (just take the same ultrafilter), and the limit case is covered by Lemma \ref{lemma equivalent constructions}. For the successor case, let $e: \Upsilon_\alpha(\mathcal{A}, \mathbf{I}) \to \mathbb{C}_\alpha(\mathcal{A}, \mathbf{I})$ be the isomorphism, and $\mathcal{U}' = \{e[X] \mid X \in \mathcal{U}_\alpha\}$. One might check $\mathcal{U}'$ is an ultrafilter over $\mathbb{C}_\alpha(\mathcal{A}, \mathbf{I})$. Consider then $F = \big{\{} \{x \in \mathfrak{C}_\alpha(A, \mathbf{I}) \mid [x] \in X\} \mid X \in \mathcal{U}' \big{\}}$. By Lemma \ref{lemma existence covarying filter}, $F$ is a filter base over $\mathfrak{C}_\alpha(A, \mathbf{I})$, and $\{x \in \mathfrak{C}_\alpha(A, \mathbf{I}) \mid [x] \in X\} \in F$ iff $X \in \mathcal{U}'$. Let $\theta: \mathbb{C}_\alpha(\mathcal{A}, \mathbf{I}) \to \mathfrak{C}_\alpha(A, \mathbf{I})$ be a choice function on the equivalence classes. By the definition of $F$, $\mathfrak{C}_\alpha(A, \mathbf{I}) \setminus \theta[\mathbb{C}_\alpha(\mathcal{A}, \mathbf{I})] \not\in F$, and we may also see $\theta[\mathbb{C}_\alpha(\mathcal{A}, \mathbf{I})] \cap X \neq \varnothing$ for any $X \in F$. Thus, $F' = F \cup \{\theta[\mathbb{C}_\alpha(\mathcal{A}, \mathbf{I})]\}$ is also a filter base. Let $\mathcal{U}'_\alpha \supseteq F'$ be an ultrafilter and let it define $\sim_{\alpha^+}$. Consider $w: \Upsilon_\alpha(\mathcal{A}, \mathbf{I})^{\Upsilon_\alpha(\mathcal{A}, \mathbf{I})} \to \mathfrak{C}_{\alpha^+}(A, \mathbf{I})$ such that for each $f \in \Upsilon_\alpha(\mathcal{A}, \mathbf{I})^{\Upsilon_\alpha(\mathcal{A}, \mathbf{I})}$, $w(f): x \mapsto \theta \big{(} e \big{(} f(e^{-1}([x]_\alpha)) \big{)} \big{)}$. For clarity, let us abbreviate $e \big{(} f(e^{-1}([j]_\alpha)) \big{)}$ as $f^e(j)$. We now show that $e': \Upsilon_{\alpha^+}(\mathcal{A}, \mathbf{I}) \to \mathbb{C}_{\alpha^+}(\mathcal{A}, \mathbf{I}); [x]_{\mathcal{U}_\alpha} \mapsto [w(x)]_{\alpha^+}$ is an isomorphism. It is trivial to see that $e'$ preserves the interpretations of constants, that is, that for any constant $\mathrm{c}$, $e'(\mathrm{c}^{\Upsilon_{\alpha^+}(\mathcal{A}, \mathbf{I})}) = e'([\overline{\mathrm{c}^{\Upsilon_\alpha(\mathcal{A}, \mathbf{I})}}]_{\mathcal{U}_\alpha}) = [w(\overline{\mathrm{c}^{\Upsilon_\alpha(\mathcal{A}, \mathbf{I})}})]_{\alpha^+} = [\overline{\mathrm{c}^{\mathbb{C}_\alpha(\mathcal{A}, \mathbf{I})}}]_{\alpha^+} = \mathrm{c}^{\mathbb{C}_{\alpha^+}(\mathcal{A}, \mathbf{I})}$. Let $e'([f]_{\mathcal{U}_\alpha}) = [w(f)]_{\alpha^+} = [w(g)]_{\alpha^+} = e'([g]_{\mathcal{U}_\alpha})$, meaning $\{j \in \mathfrak{C}_\alpha(A, \mathbf{I}) \mid w(f)(j) \sim_\alpha w(g)(j)\} \in \mathcal{U}'_\alpha$. Since $\mathcal{U}'_\alpha$ is closed upwards, we have

\begin{center}
$\{j \in \mathfrak{C}_\alpha(A, \mathbf{I}) \mid \exists k \big{(} j \sim_\alpha k\ \&\ w(f)(k) \sim_\alpha w(g)(k) \big{)}\} \in \mathcal{U}'_\alpha$.
\end{center}

\noindent
By the definition of $\mathcal{U}'_\alpha$, one might check that is the case only if

\begin{center}
$\{j \in \mathfrak{C}_\alpha(A, \mathbf{I}) \mid \exists k \big{(} j \sim_\alpha k\ \&\ w(f)(k) \sim_\alpha w(g)(k) \big{)}\} \in F$.
\end{center}

\noindent
Thus, by the definition of $F$, we have

\begin{center}
$\{[j]_\alpha \in \mathbb{C}_\alpha(A, \mathbf{I}) \mid w(f)(j) \sim_\alpha w(g)(j)\} \in \mathcal{U}'$.
\end{center}

\noindent
Since $w(f)(j) = \theta \big{(} f^e(j) \big{)} \sim_\alpha \theta \big{(} g^e(j) \big{)} = w(g)(j)$ iff $f^e(j) = e \big{(} f(e^{-1}([j]_\alpha)) \big{)} = e \big{(} g (e^{-1}([j]_\alpha)) \big{)} = g^e(j)$, we have

\begin{center}
$\{[j]_\alpha \in \mathbb{C}_\alpha(A, \mathbf{I}) \mid f^e(j) = g^e(j)\} \in \mathcal{U}'$.
\end{center}

\noindent
By the definition of $\mathcal{U}'$, that is the case only if

\begin{center}
$\{e^{-1}([j]_\alpha) \in \Upsilon_\alpha(\mathcal{A}, \mathbf{I}) \mid f^e(j) = g^e(j)\} \in \mathcal{U}_\alpha$,
\end{center}

\noindent
that is,

\begin{center}
$\{e^{-1}([j]_\alpha) \in \Upsilon_\alpha(\mathcal{A}, \mathbf{I}) \mid e \big{(} f(e^{-1}([j]_\alpha)) \big{)} = e \big{(} g(e^{-1}([j]_\alpha)) \big{)} \} \in \mathcal{U}_\alpha$.
\end{center}

\noindent
Since $e$ is an isomorphism, its inverse is surjective, and therefore we may see

\begin{center}
$\{j \in \Upsilon_\alpha(\mathcal{A}, \mathbf{I}) \mid e \big{(} f(j) \big{)} = e \big{(} g(j) \big{)}\} \in \mathcal{U}_\alpha$.
\end{center}

\noindent
But since $e$ is injective, $e \big{(} f(j) \big{)} = e \big{(} g(j) \big{)}$ iff $f(j) = g(j)$, so that

\begin{center}
$\{j \in \Upsilon_\alpha(\mathcal{A}, \mathbf{I}) \mid f(j) = g(j)\} \in \mathcal{U}_\alpha$,
\end{center}

\noindent
and therefore $[f]_{\mathcal{U}_\alpha} = [g]_{\mathcal{U}_\alpha}$, so $e'$ is injective. Let now $\mathrm{R}^{\Upsilon_{\alpha^+}(\mathcal{A}, \mathbf{I})} [f_0]_{\mathcal{U}_\alpha} ... [f_n]_{\mathcal{U}_\alpha}$. That is the case iff

\begin{center}
$\{j \in \Upsilon_\alpha(\mathcal{A}, \mathbf{I}) \mid \mathrm{R}^{\Upsilon_\alpha(\mathcal{A}, \mathbf{I})} f_0(j)...f_n(j)\} \in \mathcal{U}_\alpha$, iff\\
$\{e(j) \in \mathbb{C}_\alpha(\mathcal{A}, \mathbf{I}) \mid \mathrm{R}^{\Upsilon_\alpha(\mathcal{A}, \mathbf{I})} f_0(j)...f_n(j)\} \in \mathcal{U}'$.
\end{center}

\noindent
Since $e$ is surjective, $j = e^{-1}([k]_\alpha)$ for some $k \in \mathfrak{C}_\alpha(A, \mathbf{I})$, so that is the same as

\begin{center}
$\{[k]_\alpha \in \mathbb{C}_\alpha(\mathcal{A}, \mathbf{I}) \mid \mathrm{R}^{\Upsilon_\alpha(\mathcal{A}, \mathbf{I})} f_0 \big{(} e^{-1}([k]_\alpha) \big{)}...f_n \big{(} e^{-1}([k]_\alpha) \big{)} \} \in \mathcal{U}'$.
\end{center}

\noindent
Since $e$ is an isomorphism, that is the same as

\begin{center}
$\{[k]_\alpha \in \mathbb{C}_\alpha(\mathcal{A}, \mathbf{I}) \mid \mathrm{R}^\sim_\alpha f_0^e(k)...f_n^e (k)\} \in \mathcal{U}'$.
\end{center}

\noindent
By the definition of $F$ and $\mathcal{U}'_\alpha$, that is the case iff

\begin{center}
$\{k \in \mathfrak{C}_\alpha(A, \mathbf{I}) \mid \exists j \in \mathfrak{C}_\alpha(A, \mathbf{I}) \big{(} k \sim_\alpha j\ \&\ \mathrm{R}^\sim_\alpha f_0^e(j)...f_n^e(j) \big{)}\} \in \mathcal{U}'_\alpha$.
\end{center}

\noindent
Since $f_i^e(k) = [\theta \big{(} f_i^e(k) \big{)}]_\alpha = [w(f_i)(k)]_\alpha$, that is the same as

\begin{center}
$\{k \in \mathfrak{C}_\alpha(A, \mathbf{I}) \mid \exists j \in \mathfrak{C}_\alpha(A, \mathbf{I}) \big{(} k \sim_\alpha j\ \&\ \mathrm{R}^\sim_\alpha [w(f_0)(j)]_\alpha ...  [w(f_n)(j)]_\alpha \big{)}\} \in \mathcal{U}'_\alpha$.
\end{center}

\noindent
However, if $k \sim_\alpha j$, then $e^{-1}([k]_\alpha) = e^{-1}([j]_\alpha)$. Thus, by the definition of $w$, for any $f \in \Upsilon_\alpha(\mathcal{A}, \mathbf{I})^{\Upsilon_\alpha(\mathcal{A}, \mathbf{I})}$, if $k \sim_\alpha j$, then $f^e(k) = f^e(j)$, and so $w(f)(k) = w(f)(j)$. Therefore, the above is the case iff

\begin{center}
$\{k \in \mathfrak{C}_\alpha(A, \mathbf{I}) \mid \mathrm{R}^\sim_\alpha [w(f_0)(k)]_\alpha ... [w(f_n)(k)]_\alpha\} \in \mathcal{U}'_\alpha$,
\end{center}

\noindent
so by Lemma \ref{lemma relation quotient}, that is the case iff $\mathrm{R}^\sim_{\alpha^+} [w(f_0)]_{\alpha^+} ...  [w(f_n)]_{\alpha^+}$, that is, $\mathrm{R}^\sim_{\alpha^+} e([f_0]_{\mathcal{U}_\alpha}) ...  e([f_n]_{\mathcal{U}_\alpha})$. By a similar argument, we may show $e'$ preserves the functions. At last, to see $e'$ is surjective, it suffices to show for any $f \in \mathfrak{C}_{\alpha^+}(A, \mathbf{I})$ there is $f' \in \Upsilon_\alpha(\mathcal{A}, \mathbf{I})^{\Upsilon_\alpha(\mathcal{A}, \mathbf{I})}$ such that $f \sim_{\alpha^+} w(f')$. To see that, consider $f'$ such that for any $j \in \Upsilon_\alpha(\mathcal{A}, \mathbf{I})$, $f'(j) = e^{-1} \big{(} [f \big{(} \theta (e(j)) \big{)}]_\alpha \big{)}$. Then, for any $j \in \mathfrak{C}_\alpha(A, \mathbf{I})$,

\begin{center}
$w(f')(j) = \theta \big{(} e \big{(} f'(e^{-1}([j]_\alpha)) \big{)} \big{)} = \theta \big{(} e \big{(} e^{-1} \big{(} [f \big{(} \theta(e( e^{-1}([j]_\alpha))) \big{)}]_\alpha \big{)} \big{)} \big{)} = \theta \big{(} [f(\theta([j]_\alpha))] \big{)} \sim_\alpha f \big{(} \theta([j]_\alpha) \big{)}$.
\end{center}

\noindent
Therefore,

\begin{center}
$\{j \in \mathfrak{C}_\alpha(A, \mathbf{I}) \mid w(f')(j) \sim_\alpha f(j)\} \supseteq \theta[\mathbb{C}_\alpha(\mathcal{A}, \mathbf{I})] \in \mathcal{U}'_\alpha$,
\end{center}

\noindent
so since $\mathcal{U}'_\alpha$ is closed by supersets, $w(f') \sim_{\alpha^+} f$.
\end{proof}

Notice that, under a suitable assumption, the proof of the above result essentially goes both ways -- that is, to show an ultrapower hierarchy isomorphic to the hierarchy of $\mathbb{C}_\beta(\mathcal{A}, \mathbf{I})$'s may be constructed:

\begin{theorem}\label{theorem equivalent constructions 3}
Let $\{\mathbb{C}_\alpha(\mathcal{A}, \mathbf{I})\}_{\alpha \leq \beta}$ be generated by a family of ultrafilters $\{\mathcal{U}_\alpha\}_{\alpha < \beta}$ respectively over $\{\mathfrak{C}_\alpha(A, \mathbf{I})\}_{\alpha < \beta}$. If for each $\alpha < \beta$ there is a choice function $\theta : \mathbb{C}_\alpha(\mathcal{A}, \mathbf{I}) \to \mathfrak{C}_\alpha(A, \mathbf{I})$ on the equivalence classes such that $\theta[\mathbb{C}_\alpha(\mathcal{A}, \mathbf{I})] \in \mathcal{U}_\alpha$, then there is a family of ultrafilters $\{\mathcal{U}'_\alpha\}_{\alpha < \beta}$ respectively over the index sets $\{\Upsilon_\alpha(\mathcal{A}, \mathbf{I})\}_{\alpha < \beta}$ such that if an ultrapower hierarchy $\{\Upsilon_\alpha(\mathcal{A}, \mathbf{I})\}_{\alpha \leq \beta}$ is generated by those parameters, then for any $\gamma \leq \beta$, $\Upsilon_\gamma(\mathcal{A}, \mathbf{I}) \cong \mathbb{C}_\gamma(\mathcal{A}, \mathbf{I})$.
\end{theorem}

\begin{proof}
The same argument as for Theorem \ref{theorem equivalent constructions 2}.
\end{proof}

The above result allows for a narrower condition for the construction of isomorphic ultrapowers.

\begin{lemma}\label{lemma filters equivalences 2}
Let $A$ be a set, $\equiv$ an equivalence relation over it, $\kappa$ be the size of the smallest equivalence class of $A /\! \equiv$, and $\mathcal{U}$ be an ultrafilter over $A$. If for every choice function $\theta: (A /\! \equiv) \to A$, $\theta[A /\! \equiv] \not\in \mathcal{U}$, then $\mathcal{U}$ is $\kappa$-incomplete.
\end{lemma}

\begin{proof}
For each equivalence class $a \in A /\! \equiv$, let $\langle x^a_i \rangle_{i < |a|}$ be an enumeration of its elements. Then there are $\kappa$ choice functions $\theta_\tau$ such that $\theta_\tau ; a \mapsto x^a_\tau$. Since $\theta_\tau[A /\! \equiv] \not\in \mathcal{U}$, we have $A \setminus \theta_\tau[A /\! \equiv] \in \mathcal{U}$, but $\bigcap_{\tau < \kappa} (A \setminus \theta_\tau[A /\! \equiv]) = \varnothing \not\in \mathcal{U}$.
\end{proof}

\begin{proposition}\label{lemma equivalence classes size}
Let $\mathcal{A}$ be an infinite structure, $I$ an index set, and $\mathcal{U}$ an ultrafilter over it. For any $f, g \in A^I$, $|[f]| = |[g]|$, that is, every equivalence class in $\mathcal{A}^I / \mathcal{U}$ is equally sized.
\end{proposition}

\begin{proof}
If $|A| = 1$ the result is trivial, so suppose $|A| > 1$. We may build an injection from $[f]$ to $[g]$, and the other way around, in the following manner. Let $\langle \mathcal{P}(I), \leq \rangle$ be a total order. Let now $\langle [f], \preceq \rangle$ be a preorder on the members of the equivalence class $[f]$ such that $f$ is the least element, and for $p, q \neq f$, $p \preceq q$ iff $\{i \in I \mid p(i) = f(i)\} \leq \{i \in I \mid q(i) = f(i)\}$. Define $p \approx q$ iff $p \preceq q$ and $q \preceq p$, and write $\llbracket p \rrbracket$ for the equivalence class under $\approx$ of $p$. Let $|\{\llbracket p \rrbracket\}_{p \in |[f]|}| = \eta$, and $\langle \{\llbracket p \rrbracket\}_{p \in |[f]|}, < \rangle$ be the resulting total order on the equivalence classes under $\approx$, which we may then denote by the enumeration $\langle c_\alpha \rangle_{\alpha < \eta}$ (where each $c_\alpha$ is an equivalence class under $\approx$). For each $0 < \beta < \eta$, let $\langle d_{\beta, \alpha} \rangle_{\alpha < |c_\beta|}$ be an enumeration of the equivalence class $c_\beta$. For each $0 < \beta < \eta$ and $\alpha < |c_\beta|$, let $S_\beta$ be the collection of permutations in $A^I$ such that for each $h \in S$, $h(j) = g(j)$ if $j \in \{i \in I \mid f(i) = d_{\beta, \alpha}(i)\}$, and $h(j) \neq g(j)$ otherwise. Then, let $e : [f] \to [g]$ be a mapping such that $e(f) = g$, and which maps each $d_{\beta, \alpha}$ to a distinct element of $S_\beta$. Notice such mapping is possible since the elements $d_{\beta, \alpha}$ themselves correspond to analogous permutations, that is, $|c_\alpha| \leq |S_\beta|$. Notice also, by definition, $e(d_{\beta, \alpha}) \in [g]$ for any $0 < \beta < \eta$ and $0 \leq \alpha < |c_\beta|$. We now show $e$ is an injection. For $d_{\beta, \alpha}, d_{\delta, \gamma} \in [f]$, if $\alpha \neq \delta$, then $\{i \in I \mid f(i) = d_{\beta, \alpha}(i)\} \neq \{i \in I \mid f(i) = d_{\delta, \gamma}(i)\}$, so without loss of generality, we may assume there is $k \in I$ such that $f(k) = d_{\beta, \alpha}(k)$ but $f(k) \neq d_{\delta, \gamma}(k)$. By definition, that means $e(d_{\beta, \alpha})(k) = g(k)$ and $e_(d{\delta, \gamma})(k) \neq g(k)$, and so $e(d_{\beta, \alpha}) \neq e(d_{\delta, \gamma})$. If $\alpha = \delta$, then by definition $e(d_{\beta, \alpha})$ and $e(d_{\delta, \gamma})$ correspond to different permutations, and so are likewise distinct. Using an adequately modified construction we may build a similar injection in the other direction.
\end{proof}

\begin{corollary}\label{corollary equivalent constructions 2}
Let $\{\mathbb{C}_\alpha(\mathcal{A}, \mathbf{I})\}_{\alpha \leq \beta}$ be generated by a family of ultrafilters $\{\mathcal{U}_\alpha\}_{\alpha < \beta}$ respectively over $\{\mathfrak{C}_\alpha(A, \mathbf{I})\}_{\alpha < \beta}$, and $\kappa_\alpha$ be the size of the equivalence classes of $\mathbb{C}_\alpha(\mathcal{A}, \mathbf{I})$. If, for each $\alpha < \beta$, $\mathcal{U}_\alpha$ is $\kappa_\alpha$-complete, then there is a family of ultrafilters $\{\mathcal{U}'_\alpha\}_{\alpha < \beta}$ respectively over the index sets $\{\Upsilon_\alpha(\mathcal{A}, \mathbf{I})\}_{\alpha < \beta}$ such that if an ultrapower hierarchy $\{\Upsilon_\alpha(\mathcal{A}, \mathbf{I})\}_{\alpha \leq \beta}$ is generated by those parameters, then for any $\gamma \leq \beta$, $\Upsilon_\gamma(\mathcal{A}, \mathbf{I}) \cong \mathbb{C}_\gamma(\mathcal{A}, \mathbf{I})$.
\end{corollary}

\begin{proof}
Since $\mathcal{U}_\alpha$ is $\kappa_\alpha$-complete, by Lemma \ref{lemma filters equivalences 2}, there is a choice function $\theta : \mathbb{C}_\alpha(\mathcal{A}, \mathbf{I}) \to \mathfrak{C}_\alpha(A, \mathbf{I})$ such that $\theta[\mathbb{C}_\alpha(\mathcal{A}, \mathbf{I})] \in \mathcal{U}_\alpha$. The rest follows from Theorem \ref{theorem equivalent constructions 3}.
\end{proof}

Therefore, the above result shows that if an ultrapower hierarchy of a structure is such that each successor level is constructed by having the preceding level as its index set, then there is a hierarchy of cumulative powers with each successor level also constructed by having the preceding level as its index set which is isomorphic to it. Likewise, if the strongly cumulative power hierarchy generated with such choice of index sets quotiented by $\sim$ is such that the equivalence relation of each successor level is constructed by taking an ultrafilter $\mathcal{U}$ such that there is a choice function $\theta$ on the equivalence classes of $\mathbb{C}_\beta(\mathcal{A}, \mathbf{I})$ with $\theta[\mathbb{C}_\beta(\mathcal{A}, \mathbf{I})] \in \mathcal{U}$, then there is an ultrapower hierarchy isomorphic to it -- that is, there is an equivalent construction of the so quotiented strongly cumulative power hierarchy by the ultrapower hierarchy. The above results  thus pose the following question:

\medskip

\begin{itemize}
\item[] \textbf{Open Problem:} Let $\{\mathbb{C}_\alpha(\mathcal{A}, \mathbf{I})\}_{\alpha \leq \beta}$ be generated by a family of ultrafilters $\{\mathcal{U}_\alpha\}_{\alpha < \beta}$ respectively over $\{\mathfrak{C}_\alpha(A, \mathbf{I})\}_{\alpha < \beta}$, and $\kappa_\alpha$ be the size of the equivalence classes of $\mathbb{C}_\alpha(\mathcal{A}, \mathbf{I})$. Let also $\{\Upsilon_\alpha(\mathcal{A}, \mathbf{I})\}_{\alpha \leq \beta}$ be generated by a family of ultrafilters $\{\mathcal{U}'_\alpha\}_{\alpha < \beta}$ respectively over the index sets $\{\Upsilon_\alpha(\mathcal{A}, \mathbf{I})\}_{\alpha < \beta}$. Can there be $\gamma < \beta$ such that $\mathcal{U}'_\gamma$ is $\kappa_\gamma$-incomplete, while $\mathbb{C}_\beta(\mathcal{A}, \mathbf{I}) \cong \Upsilon_\beta(\mathcal{A}, \mathbf{I})$?
\end{itemize}

\medskip

\noindent
We conjecture that, in fact, such a restriction on the ultrafilters is not necessary, and so that the similarities between both constructions are such that the quotiented $\mathcal{F}$-hierarchies and ultrapower hierarchies constructed by having the previous levels as the index sets of the successor levels are equivalent, so that any structure obtained by one of the constructions may be obtained by the other in a straightforward manner.

Now, given an ultrapower hierarchy, notice regardless of the choice of family of index sets and ultrafilters, by  \L o\'{s}'s theorem we immediately have:

\begin{proposition}\label{proposition los}
For any $\alpha \leq \beta \in \mathbf{On}$, $\Upsilon_\alpha(\mathcal{A}, \mathbf{I}) \preceq \Upsilon_\beta(\mathcal{A}, \mathbf{I})$, and therefore $\mathcal{A} \preceq \Upsilon_\beta(\mathcal{A})$.
\end{proposition}

Of course, the properties of each level of this hierarchy depend on the choice of generating structure and ultrafilter of each level. Consider the case we want saturation of the structures.

For a set $X$, let $\mathcal{P}^\omega(X)$ denote the set of cofinite subsets of $X$.

\begin{definition}[$\kappa$-good ultrafilter]
Let $\mathcal{U}$ be an ultrafilter and $\kappa$ a cardinal. $\mathcal{U}$ is said to be \emph{$\kappa$-good} if for every $\tau < \kappa$ and monotonic function $H : \mathcal{P}^\omega(\tau) \to \mathcal{U}$ there is a multiplicative function $H' : \mathcal{P}^\omega(\tau) \to \mathcal{U}$ such that for any $x \in \mathcal{P}^\omega(\tau)$, $H'(x) \subseteq H(x)$.
\end{definition}

A filter is \emph{countably incomplete} if it is not $\aleph_1$-complete. In the context of saturation, countably incomplete $\kappa$-good ultrafilters are of interest for the following result.

\begin{proposition}[\cite{keisler1970}, p. 180]\label{proposition good saturation}
Let $\mathcal{U}$ be a countably incomplete $\kappa^+$-good ultrafilter over $I$. Then, for any structure $\mathcal{A}$, $|\mathcal{A}^I / \mathcal{U}| = 2^\kappa$ and $\mathcal{A}^I / \mathcal{U}$ is $\kappa^+ $-saturated.
\end{proposition}

Chang and Keisler proved in \cite{changkeisler} (p. 387), for any set $I$ of cardinality $\kappa \geq \aleph_0$, the existence of a countably incomplete $\kappa^+$-good ultrafilter over $I$ under the assumption of the Generalized Continuum Hypothesis (GCH). Later, Kunen improved the result by showing the existence of such an ultrafilter without GCH.

\begin{proposition}[\cite{kunen}, p. 304]\label{proposition good ultrafilter}
For any $I$ with $|I| = \kappa \geq \aleph_0$ there is a countably incomplete $\kappa^+$-good ultrafilter over $I$.
\end{proposition}

Take now the hierarchy previously presented, and let $|A| \geq \aleph_0$. By Propositions \ref{proposition good ultrafilter} and Proposition \ref{proposition good saturation}, we may therefore obtain the following:

\begin{corollary}\label{corollary saturation}
Let $|A| = \aleph_\alpha$, and for each $\alpha \in \mathbf{On}$ let $\Upsilon_{\alpha^+}(\mathcal{A}, \mathbf{I})$ be constructed by taking a countably incomplete $|\Upsilon_\alpha(\mathcal{A}, \mathbf{I})|^+$-good ultrafilter over $\Upsilon_\alpha(\mathcal{A}, \mathbf{I})$. Then, for each $\beta > 0$, $|\Upsilon_{\beta^+}(\mathcal{A}, \mathbf{I})| = \beth_{\beta^+}(|A|)$ and $\Upsilon_{\beta^+}(\mathcal{A}, \mathbf{I})$ is $\aleph_{\alpha+\beta^+}$-saturated, and for a limit $\lambda$, $|\Upsilon_\lambda(\mathcal{A}, \mathbf{I})| = sup\{2^{|\Upsilon_\alpha(\mathcal{A})|} \mid \alpha < \lambda\}$ and $\Upsilon_\lambda(\mathcal{A}, \mathbf{I})$ is $\aleph_{\alpha+\lambda}$-saturated. Particularly, if $A$ is countable, that means $\Upsilon_\beta(\mathcal{A}, \mathbf{I})$ is $\aleph_\beta$-saturated and $|\Upsilon_\beta(\mathcal{A}, \mathbf{I})| = \beth_\beta$.
\end{corollary}

\begin{corollary}
(GCH) If $A$ is countable, for each $\beta > 0$, $\Upsilon_\beta(\mathcal{A}, \mathbf{I})$ is $\aleph_\beta$-saturated and $|\Upsilon_\beta(\mathcal{A}, \mathbf{I})| = \aleph_\beta$.
\end{corollary}

Once again, given Theorem \ref{theorem equivalent constructions 1}, we have:

\begin{corollary}\label{corollary saturation 2}
Let $|A| = \aleph_\alpha$. If each ultrafilter in $\{\mathcal{U}_\alpha\}_{\alpha \in \mathbf{On}}$ defining $\sim_{\alpha^+}$ is a $|\Upsilon_\alpha(\mathcal{A}, \mathbf{I})|^+$-good ultrafilter respectively over $\{\Upsilon_\alpha(\mathcal{A}, \mathbf{I})\}_{\alpha \in \mathbf{On}}$, then for each $\beta > 0$, $\mathbb{C}_\beta(\mathcal{A}, \mathbf{I})$ is $\aleph_{\alpha+\beta^+}$-saturated, and for a limit $\lambda$, $|\mathbb{C}_\lambda(\mathcal{A}, \mathbf{I})| = sup\{2^{|\mathbb{C}_\alpha(\mathcal{A}, \mathbf{I})|} \mid \alpha < \lambda\}$ and $\mathbb{C}_\lambda(\mathcal{A}, \mathbf{I})$ is $\aleph_{\alpha+\lambda}$-saturated.
\end{corollary}

More specifically, if $\mathcal{A}$ is a real closed field, $|A| = \aleph_\alpha$, and for each $\mathcal{U}_\beta$ over $\Upsilon_\beta(\mathcal{A}, \mathbf{I})$ is countably incomplete and $|\Upsilon_\beta(\mathcal{A})|^+$-good, then each $\Upsilon_{\beta^+}(\mathcal{A}, \mathbf{I})$ -- and its corresponding $\mathbb{F}_{\beta^+}(\mathcal{A})$ -- is a $\aleph_{\alpha+\beta^+}$-saturated hyperreal field. Notice, however, that by Theorem \ref{theorem equivalent constructions 2} we may see there is always a straightforward choice of ultrafilter over $\mathfrak{C}_\beta(A, \mathbf{I})$ that makes $\mathbb{C}_{\beta^+}(\mathcal{A}, \mathbf{I})$ a hyperreal field; and furthermore, that by Theorem \ref{theorem initial fragment 2} and Corollary \ref{corollary initial fragment}, it is straightforward to make $\mathbb{C}_\lambda(\mathcal{A}, \mathbf{I})$ into a $\aleph_\lambda$-saturated hyperreal field, and likewise, $\mathbb{C}_\mathbf{On}(\mathcal{A}, \mathbf{I})$ into a maximally saturated real closed field without the need for countably incomplete good ultrafilters. In fact, consider the following notion.

\begin{definition}[Universally embedding]
An ordered field (or ordered abelian group, or ordered class) $\mathcal{A}$ is \emph{$\kappa$-universally embedding} if for each subfield (or abelian subgroup, or subclass) $\mathcal{B} \subseteq \mathcal{A}$ and extension field (or abelian group, or class) $\mathcal{B}' \supseteq \mathcal{B}$ such that $|B|, |B'| < \kappa$, there is $\mathcal{B}'' \subseteq \mathcal{A}$ such that $\mathcal{B}' \cong \mathcal{B}''$ and the isomorphism is an extension of identity on $\mathcal{A}$. 
\end{definition}

Introduced by Conway \cite{conway2000}, the surreal field $\mathbf{No}$ is a universally embedding real closed field that is a proper class. Besides the operations of multiplication and addition, the operations of division and exponentiation may be further defined \cite{conway2000}, \cite{driesehrlich2001}. The following results make the assumption of Global Choice, which we shall denote by writing its acronym \emph{GC}.

\begin{proposition}[\cite{conway2000}, p. 43]\label{proposition ehrlich 1}
(GC) $\mathbf{No}$ is the unique (up to isomorphism) $\aleph_\mathbf{On}$-universally embedding ordered field.
\end{proposition}

Alling similarly notes that $\mathbf{No}$ also generalizes $\eta_\beta$-sets in being an $\eta_\mathbf{On}$-class, and thus maximally saturated \cite{alling1985}. Given the saturation properties that may be obtained by either ultrapower or quotiented strongly cumulative power hierarchies previously observed, we may therefore investigate the similarities between $\mathbf{No}$, $\Upsilon_\mathbf{On}(\mathcal{A}, \mathbf{I})$, and $\mathbb{C}_\mathbf{On}(\mathcal{A}, \mathbf{I})$.

\begin{proposition}[\cite{ehrlich1988}, p. 10]\label{proposition ehrlich 2}
(GC) For $\beta \in \mathbf{On}$, $\mathcal{A}$ is an $\aleph_\beta$-universally embedding ordered field (or ordered abelian group, or ordered class) iff $\mathcal{A}$ is a real-closed field (or ordered abelian group, or ordered class, respectively) that is a $\eta_\beta$-set. Any such structure is unique up to isomorphism.
\end{proposition}

\begin{corollary}\label{corollary universal 1}
(GC) If $\mathcal{A}$ is an $\eta_0$ real-closed field (or ordered abelian group, or ordered class) and for each $\alpha < \beta$, $\Upsilon_{\alpha^+}(\mathcal{A}, \mathbf{I})$ is constructed by taking a countably incomplete $|\Upsilon_\alpha(\mathcal{A}, \mathbf{I})|^+$-good ultrafilter over $\Upsilon_\alpha(\mathcal{A}, \mathbf{I})$, then $\Upsilon_\beta(\mathcal{A}, \mathbf{I})$ is the unique (up to isomorphism) $\aleph_\beta$-universally embedding field (or ordered abelian group, or ordered class, respectively).
\end{corollary}

\begin{proof}
By Proposition \ref{proposition los}, $\Upsilon_\beta(\mathcal{A}, \mathbf{I})$ is a real-closed field. By Corollary \ref{corollary saturation}, $\Upsilon_\beta(\mathcal{A}, \mathbf{I})$ is $\aleph_\beta$-saturated, so by Proposition \ref{proposition saturation eta}, it is an $\eta_\beta$-set. By Proposition \ref{proposition ehrlich 2}, $\Upsilon_\beta(\mathcal{A}, \mathbf{I})$ is then an $\aleph_\beta$-universally embedding real-closed field. The argument is similar for the other cases.
\end{proof}

Given the above result, with Proposition \ref{proposition ehrlich 1} we therefore obtain the following broader result:

\begin{corollary}\label{corollary universal 2}
(GC) Let $\mathcal{A}$ be an $\eta_0$ real-closed field (or ordered abelian group, or ordered class) and for each $\beta \in \mathbf{On}$:

\begin{itemize}
\item[\emph{(1)}] let each $\Upsilon_{\beta^+}(\mathcal{A}, \mathbf{I})$ be constructed by taking $\Upsilon_\beta(\mathcal{A}, \mathbf{I})$ as the index set and a tails ultrafilter over it. Then:

\begin{itemize}
\item[\emph{(1.a)}] if $\aleph_\lambda$ is strongly inaccessible then $\Upsilon_\lambda(\mathcal{A}, \mathbf{I})$ is the unique (up to isomorphism) $\aleph_\lambda$-universally embedding field (or ordered abelian group, or ordered class, respectively);
\item[\emph{(1.b)}] $\Upsilon_\mathbf{On}(\mathcal{A}, \mathbf{I})$ is the unique (up to isomorphism) $\aleph_\mathbf{On}$-universally embedding field (or ordered abelian group, or ordered class, respectively).
\end{itemize}
\item[\emph{(2)}] let each $\mathbb{C}_{\beta^+}(\mathcal{A}, \mathbf{I})$ be constructed by taking an appropriate choice of ultrafilter over $\mathfrak{C}_\beta(\mathcal{A}, \mathbf{I})$.\footnote{That is, such as constructed in Theorem \ref{theorem equivalent constructions 2}}. Then:

\begin{itemize}
\item[\emph{(2.a)}] if $\aleph_\lambda$ is strongly inaccessible then $\mathbb{C}_\lambda(\mathcal{A}, \mathbf{I})$ is the unique (up to isomorphism) $\aleph_\lambda$-universally embedding field (or ordered abelian group, or ordered class, respectively);
\item[\emph{(2.b)}] $\mathbb{C}_\mathbf{On}(\mathcal{A}, \mathbf{I})$ is the unique (up to isomorphism) $\aleph_\mathbf{On}$-universally embedding field (or ordered abelian group, or ordered class, respectively).
\end{itemize}
\end{itemize}
\end{corollary}

That is, if $\mathcal{A}$ is an $\eta_0$ real-closed field, then there are simple choices of ultrafilters defining the ultrapower or quotiented cumulative power such that $\mathbb{C}_\mathbf{On}(\mathcal{A}, \mathbf{I}) \cong \mathbf{No} \cong \Upsilon_\mathbf{On}(\mathcal{A}, \mathbf{I})$. Particularly, that means the ultrapower, or cumulative power, hierarchy generated by the reals or the algebraic numbers quotiented in the described manner is a surreal field.

\bibliography{hypersurreals}
\bibliographystyle{plain}

\end{document}